\pdfoutput=1
\documentclass[final,3p,times]{elsarticle}
\usepackage{rotating}
\usepackage{bold-extra}
\usepackage{url}
\usepackage[strings]{underscore}
\usepackage[T1]{fontenc}
\usepackage{hyperref}
\journal{Theoretical Computer Science}
\title{The Smallest String Attractors of Fibonacci, Period-Doubling, and Thue--Morse~Words} %

\usepackage{amsthm}
\usepackage{amsmath}
\usepackage{amssymb}
\usepackage{mathtools}
\usepackage{stmaryrd} %
\usepackage{cleveref}
\usepackage{aliascnt}
\newtheorem{theorem}{Theorem}
\newcommand{\newsharedtheorem}[3]{%
  \newaliascnt{#1}{theorem}\newtheorem{#1}[#1]{#2}\aliascntresetthe{#1}%
  \crefname{#1}{#2}{#3}\Crefname{#1}{#2}{#3}}
\newsharedtheorem{lemma}{Lemma}{Lemmas}
\newsharedtheorem{definition}{Definition}{Definitions}
\newsharedtheorem{observation}{Observation}{Observations}
\newsharedtheorem{corollary}{Corollary}{Corollaries}
\newsharedtheorem{proposition}{Proposition}{Propositions}
\newsharedtheorem{remark}{Remark}{Remarks}

\usepackage[dvipsnames]{xcolor}
\graphicspath{{./figures/}}

\usepackage{xspace}
\xspaceaddexceptions{]\}}

\renewcommand{\a}{\mathtt{a}}
\renewcommand{\b}{\mathtt{b}}
\newcommand{\tm}{\mathcal{T}}
\newcommand{\bitcomp}[1]{\ensuremath{\overline{#1}}\xspace} %
\newcommand{\attof}[1]{\ensuremath{\mathcal{A}(#1)}\xspace}
\newcommand{\morphtm}{\mu}
\newcommand{\singleatt}{A}

\newcommand{\numattfib}{2^{n-4} + 2^{\lceil n/2 \rceil - 2}}

\newcommand{\lmus}{U} %
\newcommand{\rmus}{V} %

\newsharedtheorem{fact}{Fact}{Facts}

\begin{document}
\begin{frontmatter}
	\author[1]{Mutsunori Banbara}
	\ead{banbara@nagoya-u.jp}

	\author[2]{Hideo Bannai\texorpdfstring{\corref{cor1}}{}}
	\ead{bannai.h.557a@m.isct.ac.jp}

	\author[2]{Peaker Guo}
	\ead{peakerguo@gmail.com}

	\author[3]{Dominik K\"oppl}
	\ead{dkppl@yamanashi.ac.jp}

	\author[4]{Takuya Mieno}
	\ead{tmieno@uec.ac.jp}

	\author[4]{Yoshio Okamoto}
	\ead{okamotoy@uec.ac.jp}

	\cortext[cor1]{Corresponding author}
	\affiliation[1]{organization={Graduate School of Informatics, Nagoya University},
		city={Nagoya},country={Japan}}
	\affiliation[2]{organization={M\&D Data Science Center, Institute of Science Tokyo},
		city={Tokyo},country={Japan}}
	\affiliation[3]{organization={Department of Computer Science and Engineering, University of Yamanashi},
		city={Kofu},country={Japan}}
	\affiliation[4]{organization={Graduate School of Informatics and Engineering, University of Electro-Communications},
		city={Chofu},country={Japan}}

	\begin{abstract}
		A string attractor of a string $T[1..|T|]$ is a set of positions $\Gamma$ of $T$ such that any nonempty substring $w$ of $T$ has an occurrence that crosses a position in $\Gamma$, i.e., there is a position $i$ such that $w = T[i..i+|w|-1]$ and the intersection $[i,i+|w|-1]\cap \Gamma$ is nonempty.
		The size of the smallest string attractor of Fibonacci words and period-doubling words of length at least $2$
		is known to be $2$.
		The size of the smallest string attractor of the $n$th Thue--Morse word is known to be $4$ for $n \geq 4$.
		We completely characterize the set of all smallest string attractors of Fibonacci words,
		period-doubling words, and Thue--Morse words.
		For Fibonacci words, we show a recursive formula describing the
		$\numattfib$
		distinct position pairs that are the smallest string attractors of the $n$th Fibonacci word for $n \geq 7$.
		For period-doubling words,
		we show that the $n$th period-doubling word has exactly two smallest string attractors for every $n\geq 2$, and give them explicitly.
		For Thue--Morse words,
		we show a formula describing the 32 sets of four positions that are the smallest string attractors of the $n$th
		Thue--Morse word for $n\geq 6$.
		Our results show that strings with the same smallest attractor size can have a drastically different number of distinct smallest attractors.
	\end{abstract}

\end{frontmatter}

\section{Introduction}
For any string $T$, a set $\Gamma$ of positions is a {\em string attractor} (or simply attractor)
of $T$, if any nonempty substring of $T$
has an occurrence that {\em crosses} (i.e., contains) a position in $\Gamma$.
String attractors~\cite{DBLP:conf/stoc/KempaP18} are an important combinatorial notion that captures
the repetitiveness (compressiveness) of the string in the sense that
dictionary compression algorithms can be viewed as algorithms that compute
string attractors.
It is known that the size of the smallest string attractor of~$T$, denoted by $\gamma(T)$,
bounds all dictionary compression measures from below,
and is NP-hard to compute~\cite{DBLP:conf/stoc/KempaP18}.

The study of string attractors on well-known families of strings has attracted much attention.
Mantaci et al.~\cite{DBLP:conf/ictcs/MantaciRRRS19,DBLP:journals/tcs/MantaciRRRS21}
initiated the study of string attractors on standard Sturmian words (which include the Fibonacci words), and Thue--Morse words.
For any standard Sturmian word $T$, they showed that
at least one of
the sets $\Gamma_1 = \{ \eta+1,\eta+2 \}$ or $\Gamma_2= \{ |T|-\eta-3, |T|-\eta-2\}$
is a smallest string attractor of $T$,
where $\eta$ is the length of the longest proper palindromic prefix of $T[1..|T|-2]$.
For de Bruijn words of length~$n$, they showed that the smallest attractor size asymptotically approaches $n/\log n$.
They also showed an attractor of size $\Theta(\log |\mathcal{T}_n|)$
for the $n$th Thue--Morse word $\mathcal{T}_n$.
Later, Kutsukake et al.~\cite{DBLP:conf/spire/KutsukakeMNIBT20}
showed that $\gamma(\mathcal{T}_n) = 4$ for all $n\geq 4$.
In a more general setting,
Schaeffer and Shallit~\cite{schaeffer2024stringattractorsautomaticsequences} discussed string attractors of
all prefixes of automatic sequences, including the period-doubling words, Thue--Morse words, and Tribonacci words.
In particular, they showed that the size of the smallest attractor of any prefix of the period-doubling word longer than one is $2$, and the size of the smallest attractor of any prefix of length at least $25$ of the Thue--Morse word is $4$,
using the Walnut theorem prover~\cite{mousavi2021automatictheoremprovingwalnut}.

For the Fibonacci words, defined as $F_0 = \texttt{b}$, $F_1 = \texttt{a}$, $F_n = F_{n-1}F_{n-2}$,
it holds for $n \geq 3$ that $\eta = f_{n-1}-2$ (where $f_{n-1} = |F_{n-1}|$),
and the above result of Mantaci et al.
translates to the sets $\Gamma_1 = \{ f_{n-1}-1, f_{n-1} \}$ or  $\Gamma_2= \{ f_{n-2}-1, f_{n-2} \}$.
Since $\Gamma_2$ cannot be an attractor
(it can be shown that
the length $f_{n-1}$ suffix of $F_n$ occurs uniquely and does not have an occurrence crossing a position in $\Gamma_2$),
it follows that $\Gamma_1$ is an attractor of $F_n$.
Hence, this gives an explicit pair of positions to be a smallest attractor of Fibonacci words.
However, in general, a given string can have multiple smallest string attractors.
For example, any single position of a unary string is an attractor.
For a more interesting example, the set of all smallest string attractors of $F_8$ is shown in~\Cref{fig:att-example-n=8}.
See also~\Cref{fig:attractors-small-orders} for more examples.

In this paper, we take a deeper look into the combinatorial aspects of string attractors of Fibonacci words, period-doubling words,
and Thue--Morse words,
and give a complete characterization of their smallest string attractors.
Our characterization for Fibonacci words is based on occurrences of {\em singular words}~\cite{DBLP:journals/ejc/WenW94} in the Fibonacci word.
We show a recursive formula describing the set of all $\numattfib$
smallest string attractors of the $n$th Fibonacci word for all $n\geq 7$.
For the $n$th period-doubling word $D_n$
defined as $D_0 = \texttt{a}$ and $D_n = \phi(D_{n-1})$
for the morphism $\phi(\texttt{a}) =\texttt{ab}$ and $\phi(\texttt{b})=\texttt{aa}$,
we show that the position pairs
$\{3\cdot 2^{n-3},6\cdot 2^{n-3}\}$ and
$\{4\cdot 2^{n-3},6\cdot 2^{n-3}\}$
characterize the smallest attractors for $n \geq 3$.
For the $n$th Thue--Morse word
defined as $\tm_0 = \texttt{a}$ and $\tm_n = \morphtm(\tm_{n-1})$
for the morphism $\morphtm(\a) =\a\b$ and $\morphtm(\b)=\b\a$,
we show the $32$ sets of $4$ positions that are the smallest attractors for $n\geq 6$.
Our proofs for period-doubling and Thue--Morse words follow a common scheme:
every smallest attractor of the $n$th word is shown to project, by halving its positions, onto a smallest attractor of the $(n-1)$st word,
which reduces the characterization to excluding a constant number of candidate positions by means of short substrings.
To the best of our knowledge, this is the first work which shows a complete characterization of the set of smallest string attractors for non-trivial families of strings.

It is also worth noting that, although the smallest attractor size for both the Fibonacci and period-doubling words is two, our results show that the \emph{number} of smallest attractors can differ drastically, namely, from polynomial in the length to constant.
We do not yet have a clear interpretation of what this quantity captures,
and its understanding remains an open and interesting question, as it may allow us to
consider more fine-grained types of repetitiveness. Other related open problems are listed in~\Cref{sec:open_problems}.
A preliminary version of this work appeared in the proceedings of CPM 2026~\cite{DBLP:conf/cpm/BanbaraBGKMO26}.
This version further adds the full characterization of the smallest attractors of Thue--Morse words,
which was announced in the presentation of the conference version at CPM 2026 (June 2026),
as well as details of the verification of our results using the Walnut theorem prover~\cite{mousavi2021automatictheoremprovingwalnut}.

\subsection*{Related Work}
Mieno et al.~\cite{DBLP:conf/cpm/MienoIH22} gave a
complete characterization of smallest straight-line programs (SLPs) of Fibonacci words,
showing that an SLP of a Fibonacci word is a smallest SLP if and only if
it can be obtained by some implementation of the RePair algorithm~\cite{DBLP:conf/dcc/LarssonM99},
i.e.,
the recursive pairing of a most frequent adjacent symbol pair.
The number of distinct smallest SLPs for $F_{n}$ was shown to be
$2\lfloor (n+1)/2\rfloor -2$ for $n\geq 3$.

We note that for the strings we have considered,
the approach of Schaeffer and Shallit~\cite{schaeffer2024stringattractorsautomaticsequences}
and the software Walnut~\cite{mousavi2021automatictheoremprovingwalnut}
can be applied,
and it is possible to  compute a finite state automaton
for symbols in $(\{0,1\}^d)^*$, such that accepting paths spell out a representation of the attractor positions (two for Fibonacci and period-doubling words, and four for Thue--Morse words, i.e., $d=3$ or $d=5$) and the length of a prefix of the string.
This approach is discussed in~\ref{sec:Walnut}.
While the exact sets of positions that the accepting paths of these automata represent are not readily apparent from the automata themselves, by comparing them with automata constructed directly from our characterizations,
we have mechanically confirmed that our main theorems hold for all orders covered by their statements
(\ref{sec:Walnut-verification}).
However, such a verification would not seem to give us much insight into why the statements hold,
which is what our combinatorial proofs provide.
Schaeffer and Shallit~\cite{schaeffer2024stringattractorsautomaticsequences} and Cassaigne et al.~\cite{DBLP:journals/jcta/CassaigneGRRSS24}
have studied another related attractor-based parameter $\mathit{span}(w)$,
defined as the smallest distance between the leftmost and rightmost positions in an attractor of $w$.

We recently became aware of the independent and concurrent work of Faro, Marino, and Pavone~\cite{FaroMarinoPavone2026},
which obtains the same characterization of the smallest string attractors of Thue--Morse words as our~\Cref{thm:tm_smallest_attractors} (Theorem~5.4 of~\cite{FaroMarinoPavone2026}).
Their proof is different from ours: it does not use induction on the order, but derives the characterization from the sets of positions covered by the occurrences of the minimal unique substrings of three consecutive Thue--Morse words (Lemma~5.1 of~\cite{FaroMarinoPavone2026}),
building on the characterization of these minimal unique substrings by Kishi et al.~\cite{KishiGUB26} and on an exact description of their occurrences in the subsequent Thue--Morse words (Lemma~4.3 of~\cite{FaroMarinoPavone2026}).
Our proof instead follows the inductive scheme described above, which we also use for period-doubling words
(\Cref{lem:pd_attractors_completeness_1,lem:tm_attractors_completeness,lem:tm_outer_positions}).

\begin{figure}[t]
	\centering
	\includegraphics[width=\linewidth]{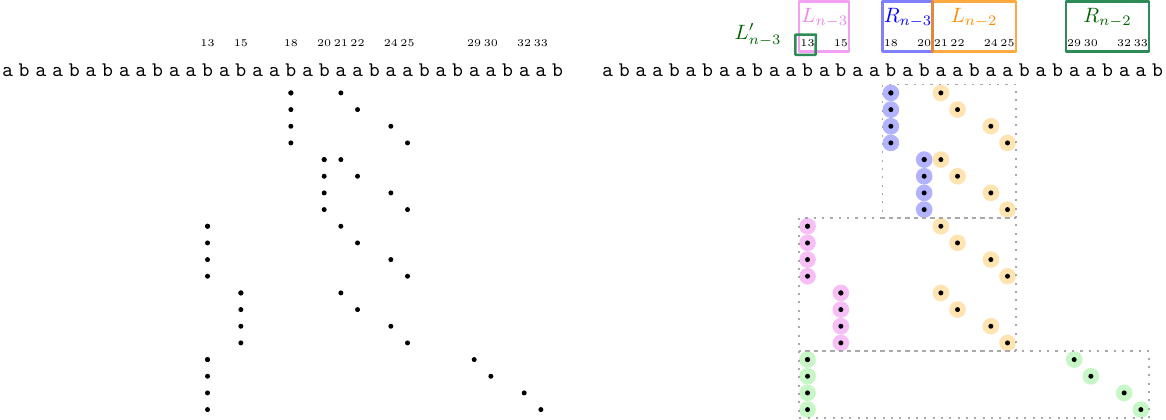}
	\caption{
		For $n=8$.
		Left: all smallest attractors of $F_n$,
		where each pair of horizontally aligned dots
		represents an attractor.
		Right: characterization of these attractors according to \Cref{thm:all_smallest_attractors_of_fib};
		the three dotted boxes (top to bottom) correspond to \Cref{lem:valid_att_from_flip,lem:valid_att_from_flip_offset,lem:valid_att_from_n-2_offset}.
		(See \Cref{fig:att-example-n=9} in \ref{sec:additional-figures} for the case $n=9$.)
	}
	\label{fig:att-example-n=8}
\end{figure}

\section{Preliminaries}
Let $\Sigma = \{\texttt{a}, \texttt{b}\}$ be a binary alphabet.
An element of $\Sigma$ is called a symbol.
An element of $\Sigma^*$ is called a string.
For a string $x$, let $|x|$ denote its length, and let $\varepsilon$ denote the empty string,
i.e., $|\varepsilon|=0$.
For any integer $1 \leq i \leq |x|$, let $x[i]$ denote the $i$th symbol of $x$,
and for any $1 \leq i\leq j \leq |x|$,
let $x[i..j] = x[i]\cdots x[j]$, and $x[i..j) = x[i..j-1]$. We also let $x[i..j] = \varepsilon$ for any $i > j$.
Let $x^R$ denote the \emph{reverse} of $x$, i.e., if $x = x[1]x[2]\cdots x[|x|]$, then $x^R = x[|x|]x[|x|-1]\cdots x[1]$.
The \emph{complement} of a symbol in $\Sigma$ is denoted as
$\bitcomp{\texttt{a}} = \texttt{b}$ and $\bitcomp{\texttt{b}} = \texttt{a}$,
and for any string $x\in \Sigma^*$,
its complement is $\bitcomp{x} = \bitcomp{x[1]}\cdots \bitcomp{x[|x|]}$.

For a set $X$ of integers and another integer $i$, let
$i\oplus X = X\oplus i := \{i + x : x \in X\}$ and
$i\ominus X := \{i - x : x \in X\}$.
For two sets $X$ and $Y$,
let $X \otimes Y := \{ \{x, y\} : (x, y) \in X \times Y \}$.
Also, for a family $\mathcal{X}$ of sets and a set $Y$, let $\mathcal{X}\otimes Y := \{X\cup\{y\} : X\in\mathcal{X},\ y\in Y\}$,
and we write $X_1\otimes X_2\otimes X_3$ for $(X_1\otimes X_2)\otimes X_3$.
For a string $T$ and an occurrence $w = T[i..j]$ of a substring $w$ of $T$,
we say that the occurrence of $w$ {\em crosses} position $k$, if $i \leq k \leq j$.
A \emph{boundary} is a pair of consecutive positions sometimes identified with
two strings ending and starting at the respective positions.
We say that an occurrence of a string crosses the boundary if it crosses both positions of the boundary.

\begin{definition}[String Attractors~\cite{DBLP:conf/stoc/KempaP18}]
	A {\em string attractor}, or an {\em attractor} for short, of a string $T$ is a set $\Gamma$ of positions of $T$
	such that for any nonempty substring $u$ of $T$, there exists an occurrence of $u$ in $T$ that crosses some position in $\Gamma$.
	We denote by $\attof{T}$ the set of all smallest attractors of $T$.
\end{definition}
It is straightforward that string attractors are invariant under string reversal, meaning that the ``mirrored'' string attractor is a string attractor of the mirrored string.

\begin{observation}\label{fact:att-rev}
	If $\Gamma$ is an attractor of a string $w$, then
	$(|w|+1)\ominus \Gamma$ is an attractor of $w^R$.
\end{observation}

A string $w$ is a {\em minimal unique substring} (MUS)~\cite{IlieS11} of $T$
if $w$ is a substring of $T$ that has a unique occurrence in $T$,
and any proper substring of $w$ has multiple occurrences in $T$.
The following observation is straightforward.
\begin{observation}\label{obs:mus_and_attractors}
	For each MUS of $T$, any string attractor of $T$ must contain at least one position that is crossed by the unique occurrence of the MUS.
\end{observation}

\section{The Smallest Attractors of Fibonacci Words}

\subsection{Properties of Fibonacci Words and Singular Words}
\begin{definition}[$F_n$ and $f_n$]
	For each $n \geq 0$,
	the $n$th \emph{Fibonacci word} is defined as
	$F_0 = \texttt{b}, F_1 = \texttt{a}$, and
	$F_n = F_{n-1} \ F_{n-2}$ for each $n \geq 2$.
	The  $n$th \emph{Fibonacci number} is defined as
	$f_0 = 1, f_1 = 1$, and $f_n = f_{n-1} + f_{n-2}$ for each  $n \geq 2$.
	Note that $f_n = |F_n|$.
\end{definition}
Let $F_\infty$ denote the infinite Fibonacci word, of which $F_n$ is a prefix for every $n\geq 1$.

\begin{definition}[$G_n$, $\Delta_n$]
	For $n \geq 2$,
	let $G_n = F_n[1..f_n-2]$ and
	$F_n = G_n\Delta_n$.
	For convenience, let $\Delta_0 = \texttt{ab}$ and $\Delta_1=\texttt{ba}$.
\end{definition}
Below are known or simple observations concerning Fibonacci words.
\begin{observation}[e.g.\ \cite{DBLP:journals/ipl/Luca81,DBLP:journals/siamcomp/KnuthMP77,Lothaire_2002}]\label{lem:fib_properties}
	The following propositions hold.
	\begin{enumerate}
		\item\label{lem:fib_properties_one}
		      For $i \geq 2$, $\Delta_i = \Delta_{i\bmod 2}$.
		\item\label{lem:fib_properties_two}
		      For $i\geq 2$, $G_i$ is a palindrome
		      (for $i \geq 5$ by $G_i = G_{i-2}\Delta_{i-2}G_{i-3}\Delta_{i-3}G_{i-2}$ and induction; $G_2 = \varepsilon$, $G_3 = \texttt{a}$, and $G_4 = \texttt{aba}$ are palindromes).
		\item\label{lem:fib_properties_four}
		      For $i \geq 3$, $F_i = F_{i-1}F_{i-2} = F_{i-2}G_{i-1}\Delta_{i}$.
		\item\label{lem:fib_properties_three}
		      For $i \geq 5$,
		      \begin{align*}
			      F_i & = G_i\Delta_i = F_{i-1}G_{i-2}\Delta_{i}
			      = G_{i-2}\Delta_{i-2}G_{i-3}\Delta_{i-3}G_{i-2}\Delta_{i}
			      = G_{i-2}(F_{i-3})^R(F_{i-2})^R\Delta_{i}      \\
			          & = G_{i-2}(F_{i-1})^R\Delta_i.
		      \end{align*}
	\end{enumerate}
\end{observation}

\begin{definition}[Singular Words]\label{def:singular_words}
	Let $S_{-1} = \varepsilon$, $S_0 = \texttt{a}$, $S_1 = \texttt{b}$, and $S_n = S_{n-2}S_{n-3}S_{n-2}$ for $n \geq 2$.
\end{definition}
Notice that $|F_{n}|=|S_{n}|$,
and all singular words are palindromes.
In the literature~\cite{DBLP:journals/ejc/WenW94,DBLP:journals/jintseq/Fici15}, these words
have been defined to be $S_n = F_n[f_n-1]F_n[1..f_n-1] = F_n[f_n-1]G_nF_n[f_n-1]$.
However, we start from the above definition
(Property 2 (4) in~\cite{DBLP:journals/ejc/WenW94})
to simplify the presentation.
The next lemma summarizes results shown in~\cite{DBLP:journals/ejc/WenW94} that we utilize.
\begin{lemma}[\cite{DBLP:journals/ejc/WenW94}]\label{lem:s_properties}
	For any $n \ge 0$, the following propositions hold.
	\begin{enumerate}[(1)] %
		\item\label{lem:s_properties_one}
		      $S_{n} = S_{n \bmod 2}\left(\prod_{i=0}^{n-2}S_i\right)$. %
		      \hfill (Property 2 (12) in~\cite{DBLP:journals/ejc/WenW94})
		\item\label{lem:s_properties_two}
		      $S_n$ is not a substring of $S_{n+1}$ nor of $\prod_{i=0}^{n-1}S_i$. %
		      \hfill (Properties 2 (2) and (13) in~\cite{DBLP:journals/ejc/WenW94})
		\item\label{lem:s_properties_three}
		      Let $S_n S_{n+1} = u_1u_2u_3$ where $u_1$ is a non-empty proper prefix of $S_n$ and $u_3$ is a non-empty proper suffix of $S_{n+1}$. Then, $u_2 \neq S_i$ for any $i\geq 0$. %
		      \hfill{(Lemma 3 in~\cite{DBLP:journals/ejc/WenW94})}
	\end{enumerate}
\end{lemma}

The following lemma is essentially the same
as what has been shown by Wen and Wen~(Theorem 1 in~\cite{DBLP:journals/ejc/WenW94})
for the infinite Fibonacci string.

\begin{lemma}\label{lem:fib_by_s}
	For any $n \geq 1$, $F_n = \left(\prod_{i=0}^{n-2}S_i\right)S_{(n-1)\bmod 2}$.
\end{lemma}
\begin{proof}
	Induction on $n$. The equation holds for $n=1$ and $n=2$. For $k\geq 2$, suppose it holds for all $1\leq n \leq k$.
	Then,
	\begin{align*}
		F_{k+1} = F_{k}F_{k-1}
		= \left(\prod_{i=0}^{k-2}S_i\right)S_{(k-1)\bmod 2} \left(\prod_{i=0}^{k-3}S_i\right)S_{(k-2)\bmod 2}
		= \left(\prod_{i=0}^{k-1}S_i\right)S_{k\bmod 2}
	\end{align*}
	where the last equation uses
	$S_{(k-1)\bmod 2} \left(\prod_{i=0}^{k-3}S_i\right) = S_{k-1}$ by~\Cref{lem:s_properties}~(\ref{lem:s_properties_one}).
	Therefore, the equation holds for $n=k+1$.
\end{proof}

\begin{corollary}\label{obs:s_i_first_occurrence}
	For any $k\geq 0$,
	$S_k$ occurs uniquely in $S_0\cdots S_k$ and
	the first (leftmost) occurrence of $S_k$ in $F_\infty$ spans the range $[f_{k+1}..f_{k+2}-1]$.
\end{corollary}
\begin{proof}
	$S_k$ cannot occur in $\prod_{i=0}^{k-1}S_i$ nor
	contain $S_{k-1}$ (\Cref{lem:s_properties}~(\ref{lem:s_properties_two})),
	nor cross the boundary of $S_{k-1}S_k$ (\Cref{lem:s_properties}~(\ref{lem:s_properties_three})).
	The range follows from $|S_k| = f_k$ and $\sum_{i=0}^{k-1}f_i = f_{k+1}-1$.
\end{proof}
\begin{lemma}\label{lem:mus_of_fib}
	For any $n \geq 5$,
	$S_{n-3}$ and $S_{n-2}$ are the (only) MUSs of $F_n$.
\end{lemma}
\begin{proof}
	$S_{n-3}$ occurs uniquely in
	$F_n=\left(\prod_{i=0}^{n-2}S_i \right)S_{(n-1)\bmod 2}$.
	It cannot occur after its first occurrence (characterized by~\Cref{obs:s_i_first_occurrence}) since
	it cannot occur in $S_{n-2}$ (\Cref{lem:s_properties}~(\ref{lem:s_properties_two}))
	nor cross the boundary of
	$S_{n-3}S_{n-2}$ as well as the boundary of $S_{n-2}S_{(n-1)\bmod 2}$~(\Cref{lem:s_properties}~(\ref{lem:s_properties_three}); for the second boundary, note that $S_{n-2}S_{(n-1)\bmod 2}$ is a prefix of $S_{n-2}S_{n-1}$ by \Cref{lem:s_properties}~(\ref{lem:s_properties_one})).
	Similarly, $S_{n-2}$ occurs uniquely in $F_n$, since no occurrence of $S_{n-2}$ crosses the boundary of $S_{n-2}S_{(n-1)\bmod 2}$.
	We have
	$S_{n-3} = F_n[f_{n-2}..f_{n-1}-1]$, $S_{n-2} = F_n[f_{n-1}..f_{n}-1]$
	from~\Cref{obs:s_i_first_occurrence},
	and any substring containing them are unique.
	Since
	$F_n = G_{n-1}\Delta_{n-1}F_{n-2} = F_{n-2}G_{n-1}\Delta_{n}$ by Observation~\ref{lem:fib_properties},
	it holds that
	$F_n[1.. f_{n-1}-2] = F_n[f_{n-2}+1.. f_{n}-2] = G_{n-1}$
	and $F_n[f_{n-1}+1..f_n] = F_{n-2}$ are repeating substrings.
	These are the maximal substrings of $F_n$ not containing $S_{n-3}$ nor $S_{n-2}$.
	All substrings of these maximal substrings are repeating as well, and include all proper substrings of $S_{n-3}$ and $S_{n-2}$. Therefore $S_{n-3}$ and $S_{n-2}$ are MUSs.
	Also, since a substring containing a MUS must be unique,
	there can be no other MUSs in~$F_n$.
\end{proof}

\begin{definition}[$\lmus_n$ and $\rmus_n$]\label{def:mus_rng}
	For each $n \geq 5$, define
	$\lmus_n = [f_{n-2} .. f_{n-1}-1]$ and $\rmus_n = [f_{n-1} .. f_{n}-1]$, i.e., the ranges of positions crossed by the MUSs
	$S_{n-3}$ and $S_{n-2}$, respectively.
\end{definition}

\begin{proposition}[\textnormal{\citep[Thm.~5]{DBLP:conf/ictcs/MantaciRRRS19}}]\label{prop:fib_smallest_attractor_size}
	The size of the smallest string attractor of $F_n$ is $2$ for $n \geq 2$.
\end{proposition}

For any $n\geq 5$, a Fibonacci word has exactly two MUSs
and the ranges of positions crossed by their occurrence are disjoint
(\Cref{obs:s_i_first_occurrence},~\Cref{lem:mus_of_fib},~\Cref{def:mus_rng}).
Since an occurrence of a unique substring must contain an attractor position
(\Cref{obs:mus_and_attractors})
and from~\Cref{prop:fib_smallest_attractor_size},
the smallest attractor must consist of one position from each of the two MUS ranges.
\begin{observation}\label{obs:att_mus_rng}
	For $n\geq 5$, $\attof{F_n} \subseteq \lmus_n \otimes \rmus_n$.
\end{observation}

In~\Cref{subsec:invalid_attractors}, we first characterize a subset
of position pairs in $\lmus_n \otimes \rmus_n$ that cannot be an attractor of $F_n$
by arguments based on the occurrence of singular words in $F_n$.
We then prove in~\Cref{subsec:valid_attractors} that all other position pairs are
attractors of $F_n$.

\subsection{Invalid Attractor Position Pairs}\label{subsec:invalid_attractors}
Here, we characterize position pairs in $\lmus_n \otimes \rmus_n$ that cannot be attractors of $F_n$ by using occurrences of singular words in $F_n$.
We first show that all occurrences of singular words $S_i$ in $F_n$ can be captured
by their occurrences in our recursive definition~(\Cref{def:singular_words}).
We interpret the recursion as a grammar rule,
and show that each occurrence of singular words in $F_n$
corresponds to a node in the parse tree of this grammar,
where a node in the parse tree is considered to span exactly the interval of positions corresponding to the substring it derives.
	{See Figure~\ref{fig:singular_parse_tree} for a concrete example.}

\begin{figure}[t]
	\centering
	\includegraphics[width=\linewidth]{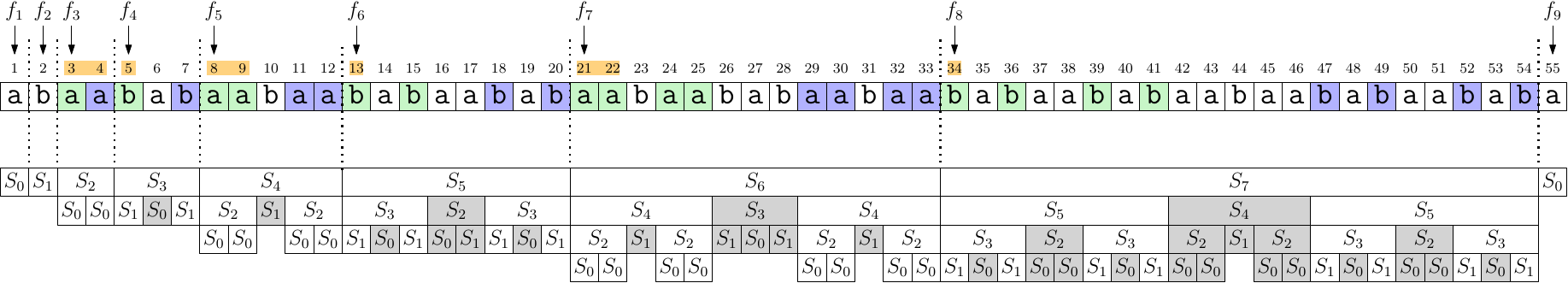}
	\caption{
		Illustration of \Cref{lem:s_parse_tree} and \Cref{def:l_r_l_prime}.
		Top: $F_9$.
		Bottom: the parse tree of $F_9$ based on singular words.
		For $2 \leq k \leq 7$:
		the gray blocks are descendants of the center child $S_{k-3}$ in the parse tree for $S_k$, i.e. $S_k = S_{k-2}S_{k-3}S_{k-2}$.
		The green and blue blocks correspond to positions in $L_k$ and $R_k$, respectively.
		(Notice that these positions are not crossed by any center child in the parse tree for $S_k$.)
		The positions corresponding to $L'_k$ are highlighted directly in orange.
	}
	\label{fig:singular_parse_tree}
\end{figure}

\begin{lemma}\label{lem:s_parse_tree}
	Consider the equations of~\Cref{lem:fib_by_s} and~\Cref{def:singular_words} as grammar rules that derive
	$F_n$ for $n \geq 1$, i.e.,
	$\{F_n \rightarrow \left(\prod_{i=0}^{n-2} S_i\right) S_{(n-1)\bmod 2},
		S_{-1} \rightarrow \varepsilon,
		S_{0} \rightarrow \texttt{a}$, $S_{1} \rightarrow \texttt{b}\}\cup
		\{ S_i \rightarrow S_{i-2}S_{i-3}S_{i-2} \mid i = 2, \ldots, n-2\}$.
	Then, for any $i\geq 0$, every occurrence of $S_i$ in $F_n$ corresponds to a node in the parse tree of this grammar.
\end{lemma}
\begin{proof}
	Due to~\Cref{obs:s_i_first_occurrence},
	there is no occurrence of $S_i$ in $F_n$ before its first (leftmost) occurrence in a node of the parse tree.
	Suppose there is a later occurrence of $S_i$ that does not correspond to a node in the parse tree of the grammar,
	and consider the lowest common ancestor (LCA) in the parse tree of the leftmost position and rightmost position of the occurrence. From the above observation, if the LCA is $F_n$, this implies that an occurrence of $S_i$ crosses a boundary of $S_kS_{k+1}$ for some $k \geq i$ (or the boundary of $S_{n-2}S_{(n-1)\bmod 2}$, which is a prefix of $S_{n-2}S_{n-1}$ by \Cref{lem:s_properties}~(\ref{lem:s_properties_one})), which violates~\Cref{lem:s_properties}~(\ref{lem:s_properties_three}).
	Therefore, the LCA must be $S_k$ for some $k > i$.
	Furthermore, since $S_i$ does not occur in $S_{i+1}$~(\Cref{lem:s_properties}~(\ref{lem:s_properties_two})),
	the LCA must be $S_k$ for some $k > i+1$, thus, $i \leq k-2$.
	By definition,
	$S_k = S_{k-2}\cdot S_{k-3}\cdot S_{k-2}$.
	Therefore,
	$S_i$ must have an occurrence in $S_k$
	(i) containing the center child $S_{k-3}$ (only when $i = k-2$), (ii) crossing the boundary of $S_{k-2}$ and $S_{k-3}$ but not containing $S_{k-3}$, or (iii) crossing the boundary of $S_{k-3}$ and $S_{k-2}$ but not containing $S_{k-3}$.
	Case (i) cannot happen since $S_{k-3}$ is not a substring of $S_{k-2}$ (\Cref{lem:s_properties}~(\ref{lem:s_properties_two})).
	Case (iii) directly violates~\Cref{lem:s_properties}~(\ref{lem:s_properties_three}) %
	and Case (ii) also violates~\Cref{lem:s_properties}~(\ref{lem:s_properties_three}) %
	due to the fact that $S_i$ is a palindrome
	and the statement of~\Cref{lem:s_properties}~(\ref{lem:s_properties_three})
	holds for $S_{k}S_{k-1}$ as well.
\end{proof}
In the parse tree of $F_n$ based on the grammar defined in the statement of~\Cref{lem:s_parse_tree},
we say that the occurrence of $S_i$ is a {\em center child} of $S_{i+3}$ if it is derived by the rule $S_{i+3} \rightarrow S_{i+1}S_iS_{i+1}$.
Note that $S_2 \rightarrow S_0S_0$ has no (or an empty) center child.
We say a position $p$ is crossed by a node in the parse tree if the occurrence of the string derived by the node crosses $p$.

\begin{lemma}\label{lem:invalidate_center}
	Let $n \geq 5$ and $\{u,v\} \in \lmus_n \otimes \rmus_n$.
	If $u$ or $v$ is crossed by a center child in the parse tree of some $S_i$, then $\{u,v\}$ is not an attractor of $F_n$.
\end{lemma}

\begin{proof} See~\Cref{fig:invalidate_center}.
	$\lmus_n$ and $\rmus_n$ correspond to the ranges of positions crossed by the MUSs $S_{n-3}$ and $S_{n-2}$ of $F_n$, respectively (\Cref{def:mus_rng}).
	Therefore, an attractor of size $2$ of $F_n$ must contain exactly one position from each of the ranges.
	By definition,
	$S_{n-3} = S_{n-5}S_{n-6}S_{n-5}$
	and
	$S_{n-2} = S_{n-4}S_{n-5}S_{n-4}$.
	Here, $S_{n-4}$ only occurs twice as children of $S_{n-2}$,
	and does not occur elsewhere in the strings $S_{n-3}$ and $S_{n-2}$~(\Cref{lem:s_parse_tree}).
	Therefore, an attractor position must be crossed by one of the two occurrences of $S_{n-4}$, and
	any position crossed by the center child $S_{n-5}$ of $S_{n-2}$ cannot be chosen as an attractor position,
	since it would not be crossed by $S_{n-4}$.
	Similarly, there are three occurrences of $S_{n-5}$; two as children of $S_{n-3}$ and one as a center child of $S_{n-2}$.
	However, since we cannot choose an attractor position in the center child of $S_{n-2}$,
	an attractor position in $S_{n-3}$ must be crossed by one of the two occurrences of $S_{n-5}$.
	Therefore, a position crossed by the center child $S_{n-6}$ of $S_{n-3}$ cannot be chosen as an attractor position.
	The argument can be repeated recursively.
	For $j = 1,\ldots,\lceil n/2\rceil-2$,
	given that the attractor positions must be crossed by occurrences of
	$S_{n-2j-1} = S_{n-2j-3} S_{n-2j-4} S_{n-2j-3}$ and
	$S_{n-2j} = S_{n-2j-2} S_{n-2j-3} S_{n-2j-2}$
	that were not center children in the subtrees of the parse trees of $S_{n-3}$ and $S_{n-2}$, respectively,
	attractor positions in the respective center children $S_{n-2j-4}$ and $S_{n-2j-3}$ cannot be chosen.
\end{proof}

\begin{figure}[t]
	\centering
	\includegraphics[width=\linewidth]{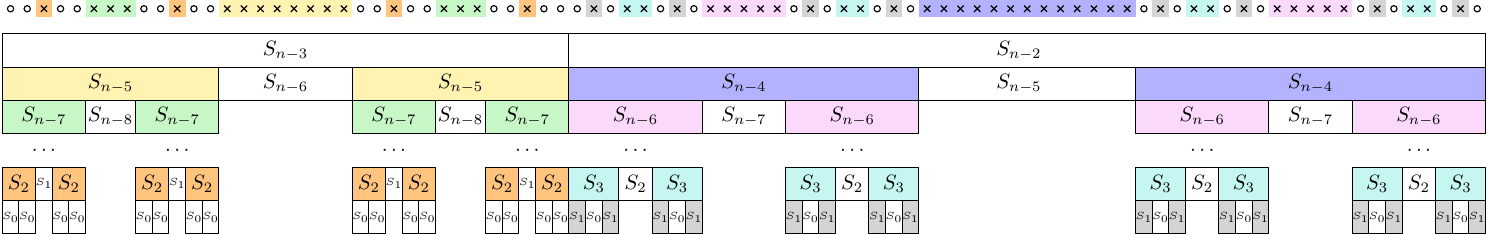}
	\caption{
		Illustration of the proof of \Cref{lem:invalidate_center} for $n=11$.
		Circles and crosses at the top mark valid and invalid attractor positions, respectively.
		A cross highlighted in a given color indicates that the position is invalid because it fails to be crossed by the substring shown in the same color.
		(Note that some combinations of the positions marked by circles in $S_{n-3}$ and $S_{n-2}$ are still invalid; these are characterized in \Cref{lem:invalidate_second_half_of_R}.)
	}
	\label{fig:invalidate_center}
\end{figure}

\begin{lemma}\label{lem:invalidate_second_half_of_R}
	Let $n \geq 7$ and $\{u,v\} \in \lmus_n \otimes \rmus_n$.
	If $u \in \lmus_n\setminus \{\min \lmus_n, \min \lmus_n+1\}$ and $v\in \rmus_n$ is in the right child of the first occurrence of $S_{n-2}$, then $\{u,v\}$ is not an attractor of $F_n$.
\end{lemma}
\begin{proof}
	See~\Cref{fig:invalid-att-2nd-half-r}.
	Consider the occurrences of
	$S_{n-4}(S_{n-3}[1..2])$ in $F_n$.
	Due to~\Cref{lem:s_parse_tree}, $S_{n-4}$ has only three occurrences in $S_{n-4}S_{n-3}S_{n-2}S_{(n-1)\bmod 2}$, and thus in $F_n$,%
	two of which are the left and right child of $S_{n-2}=S_{n-4}S_{n-5}S_{n-4}$.
	Since the last occurrence is only followed by a single symbol $S_{(n-1)\bmod 2}$,
	it follows that there can only be two occurrences of $S_{n-4}(S_{n-3}[1..2])$ in~$F_n$.
	Therefore,
	for an attractor position to be crossed by $S_{n-4}(S_{n-3}[1..2])$,
	if one of the first two positions in $\lmus_n$ (corresponding to the first two positions of $S_{n-3}$)
	is not chosen as an attractor position,
	then, we cannot choose a position in the right child of $S_{n-2}$ for the other attractor position.
	Here, notice that we require $n\geq 7$, since otherwise, the second occurrence of $S_{n-4}(S_{n-3}[1..2])$ crosses
	the boundary of the center and right children of $S_{n-2}$.
\end{proof}

\begin{figure}[t]
	\centering
	\includegraphics[width=0.72\linewidth]{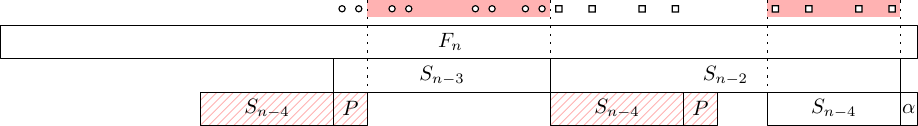}
	\caption{
		Illustration of the proof of \Cref{lem:invalidate_second_half_of_R} for $n=9$,
		where $P = S_{n-3}[1..2]$ and $\alpha = S_{(n-1)\bmod 2}$.
		Circles (resp.\ squares) at the top mark positions in $\lmus_n$ (resp.\ $\rmus_n$) that are not ruled out by \Cref{lem:invalidate_center}.
		A circle-square pair in $\lmus_n \otimes \rmus_n$ is invalid if they are both in a red region
		since they fail to be crossed by $S_{n-4} \ P$.
		(All other circle-square pairs are indeed valid, which we prove in \Cref{subsec:valid_attractors}.)
	}
	\label{fig:invalid-att-2nd-half-r}
\end{figure}

\subsection{Valid Attractor Position Pairs}\label{subsec:valid_attractors}
In this subsection, we
show that the position pairs in $\lmus_n\otimes \rmus_n$ that are not invalidated
by~\Cref{lem:invalidate_center} nor~\Cref{lem:invalidate_second_half_of_R}
are indeed attractors of $F_n$.

\begin{definition}[$L_k$, $R_k$, and $L'_k$]\label{def:l_r_l_prime}
	For $k \geq 2$,
	we define three subsets of positions within the first occurrence of $S_k$;
	let $u_k$ be the node in the parse tree corresponding to this occurrence.
	Let $L_k$ be the set of positions in the left child of $u_k$ that are \underline{not} crossed by any center child.
	Similarly, let $R_k$ be the set of positions in the right child of $u_k$ that are \underline{not} crossed by any center child.
	Further, let $L'_k$ be the set of positions corresponding to
	the occurrence of $S_{2-(k \bmod 2)}$
	in the leftmost path in the parse tree rooted at $u_k$ ($u_k$ itself when $u_k=S_2$).
\end{definition}
For example,
$L_6 = \{21,22,24,25\}$,
$R_6 = \{ 29,30,32,33\}$,
$L'_6 = \{ 21, 22 \}$.
See~\Cref{fig:singular_parse_tree}.

\begin{observation}
	$L_2 = \{ 3 \}$, $R_2 = \{ 4 \}$, $L'_2 = \{ 3, 4 \} $,
	$L_3 = \{ 5 \}$, $R_3 = \{ 7 \}$, $L'_3 = \{ 5 \} $, and
	\begin{align}
		L_k  & = (L_{k-2}\cup R_{k-2}) \oplus f_{k},\label{eqn:l_n}                        \\
		R_k  & = (L_{k-2}\cup R_{k-2}) \oplus f_{k+1} = L_k \oplus f_{k-1},\label{eqn:r_n} \\
		L'_k & = L'_{k-2}\oplus f_k\label{eqn:l_n_prime_recursion}.
	\end{align}
	for $k \geq 4$.
	For $k \geq 2$,
	\begin{align}
		L'_k =\begin{cases}
			      \{ f_{k+1}, f_{k+1}+1 \} & \text{~if $k$ is even}, \\
			      \{ f_{k+1} \}            & \text{~if $k$ is odd}.
		      \end{cases}\label{eqn:l_n_prime}
	\end{align}
\end{observation}

\begin{proof}
	Equations~(\ref{eqn:l_n})~and~(\ref{eqn:l_n_prime_recursion}) hold because, by \Cref{lem:fib_by_s},
	the offset between the starting positions of the first occurrences of $S_{k-2}$ and $S_k$ in the factorization
	$
		F_n = S_0  S_1  \cdots \underline{S_{k-2}}  S_{k-1}  \underline{S_k}  \cdots S_{n-2}  S_{(n-1) \bmod 2}
	$
	is $|S_{k-2}S_{k-1}|=f_k$.
	\Cref{eqn:l_n_prime} follows directly from \Cref{eqn:l_n_prime_recursion}.
	\Cref{eqn:r_n} holds because, in $S_k = S_{k-2}S_{k-3}S_{k-2}$, the offset between the starting positions of the two occurrences of $S_{k-2}$ is $|S_{k-2}S_{k-3}| = f_{k-1}$.
\end{proof}

The following is straightforward from the definitions and the above observation.
\begin{observation}\label{obs:l_r_min_max}
	For $k\geq 2$,
	$\min L_k = f_{k+1}$,
	$\max L_k = 2f_{k}-1$,
	$\min R_k = f_{k+1}+f_{k-1}$, and
	$\max R_k = f_{k+2}-1$.
\end{observation}

The following holds since all singular words are palindromes and thus the parse tree rooted at each singular word is symmetric, implying that
for any $k\geq 2$, $L_k$ and $R_k$ are a ``mirrored'' image of each other with respect to some midpoint.
\begin{observation}\label{obs:l_r_symmetric}
	For any $k\geq 2$, %
	$R_k = (f_{k+3}-1)\ominus L_k$,
	$L_{k+2} = (f_{k+4}-1)\ominus (L_{k}\cup R_{k})$.
\end{observation}

Using the sets of \Cref{def:l_r_l_prime}, the results of \Cref{subsec:invalid_attractors}
can be summarized as follows.

\begin{corollary}\label{cor:attoffib_contained_in_formula}
	For $n \geq 7$,
	$\attof{F_n} \subseteq ((L_{n-3}\cup R_{n-3})\otimes L_{n-2})\cup(L'_{n-3}\otimes R_{n-2})$.
\end{corollary}

\begin{proof}
	\Cref{lem:invalidate_center} implies that an attractor position pair must be
	in $(L_{n-3}\cup R_{n-3})\otimes (L_{n-2}\cup R_{n-2})$,
	where $L_{n-3}\cup R_{n-3}\subseteq \lmus_n$ and $L_{n-2}\cup R_{n-2} \subseteq \rmus_n$.
	Furthermore,~\Cref{lem:invalidate_second_half_of_R} implies
	that if one of the attractor positions is in $R_{n-2}$, then the other position must be
	in the set $\{\min \lmus_{n}, \min \lmus_{n}+1\} = \{f_{n-2}, f_{n-2}+1 \}$.
	If $n-3$ is even, this set is exactly $L'_{n-3}$.
	If $n-3$ is odd, $L'_{n-3}$ only includes $f_{n-2}$, but then $f_{n-2}+1$ is crossed by a center child (the center child $S_0$ of the node $S_3$ on the leftmost path of the parse tree of $S_{n-3}$) and cannot be an attractor position by~\Cref{lem:invalidate_center}.
	Therefore, the other position will be in $L'_{n-3}$.
	These imply that $\attof{F_n} \subseteq ((L_{n-3}\cup R_{n-3})\otimes L_{n-2})\cup(L'_{n-3}\otimes R_{n-2})$.
\end{proof}

\noindent
Below, we show that all position pairs in this set are actually attractors of $F_n$.

\begin{theorem}\label{thm:all_smallest_attractors_of_fib}
	For any $n \geq 7$,
	$\attof{F_n} = ((L_{n-3}\cup R_{n-3})\otimes L_{n-2})\cup(L'_{n-3}\otimes R_{n-2})$.
\end{theorem}
\begin{remark}\label{note:Lprime_range}
	Since $L'_{n-3}\subseteq L_{n-3}$ for $n \geq 6$,~\Cref{thm:all_smallest_attractors_of_fib}
	implies $L'_{n-3}\otimes(L_{n-2}\cup R_{n-2})\subseteq \attof{F_n}$.
\end{remark}

Since~\Cref{cor:attoffib_contained_in_formula} establishes
$\attof{F_n} \subseteq ((L_{n-3}\cup R_{n-3})\otimes L_{n-2})\cup(L'_{n-3}\otimes R_{n-2})$,
we prove $\attof{F_n} \supseteq ((L_{n-3}\cup R_{n-3})\otimes L_{n-2})\cup(L'_{n-3}\otimes R_{n-2})$
using the following three lemmas.

\begin{lemma}\label{lem:valid_att_from_flip}
	For $n \geq 8$,
	if $(L_{n-4}\cup R_{n-4})\otimes L_{n-3}\subseteq \attof{F_{n-1}}$, then
	$R_{n-3}\otimes L_{n-2} \subseteq \attof{F_n}$.
\end{lemma}

\begin{proof} See~\Cref{fig:valid-att-from-flip}.
	For any $\{q, p\} \in (L_{n-4}\cup R_{n-4}) \otimes L_{n-3}\subseteq \attof{F_{n-1}}$ with $q < p$,
	let $p' = f_{n}-p-1$ and $q' = f_{n}-q-1$.
	Since $p\in L_{n-3}$ and $q\in L_{n-4}\cup R_{n-4}$,
	$p' = (f_{n}-1)-p \in R_{n-3}$ and $q' = (f_n-1)-q \in L_{n-2}$ by~\Cref{obs:l_r_symmetric}.
	Thus, $\{p',q'\} \in R_{n-3}\otimes L_{n-2}$.
	Since the mapping $\{ p,q \}\mapsto \{p', q'\}$ is bijective,
	it suffices  to
	show that $\{p', q'\}\in \attof{F_n}$.

	From \Cref{lem:fib_properties},
	$G_n=F_{n-1}G_{n-2} = G_{n-2}(F_{n-1})^R$
	is a prefix of $F_n$.
	Essentially, $p',q'$ are positions in $F_n$
	that correspond to the positions $p,q$ in $F_{n-1}$,
	mapped to the occurrence of $(F_{n-1})^R$,
	i.e., $p' = |G_n|-p+1 = f_n-2-p+1 = f_n - p - 1$
	and $q' = |G_n|-q+1 = f_n-2-q+1 = f_n - q - 1$.
	Thus, by \Cref{fact:att-rev}, all substrings of $(F_{n-1})^R$ must have an occurrence that crosses $p'$ or $q'$.
	In order to prove that $\{p',q'\}$ is an attractor of $F_n$,
	consider any substring $u = F_n[i .. j]$ that does not cross $p'$ nor $q'$.
	Notice that since $p \geq \min L_{n-3} = f_{n-2}$, we have $p' \leq f_n - f_{n-2} - 1 = f_{n-1}-1$.
	Also, since $q \leq \max R_{n-4}=f_{n-2}-1$, we have $q' \geq f_n - (f_{n-2}-1)-1 = f_{n-1}$.
	\begin{enumerate}
		\item\label{flip-case:in-g-n-1} If $j < p'$, then $j \leq f_{n-1}-2$, so $u$ is a substring of $G_{n-1}$, which is a substring of $(F_{n-1})^R$.
		\item If $p' < i \leq j < q'$, then $u$ is trivially a substring of $(F_{n-1})^R$.
		\item\label{flip-case:in-f-n-2} If $q' < i$, then $i \geq f_{n-1}+1$ so $u$ is a substring of $F_{n-2}$, which is a substring of $(F_{n-1})^R$
		      because $(F_{n-1})^R = (\Delta_{n-1})^RG_{n-1} = (\Delta_{n-1})^R F_{n-2}G_{n-3}$.
	\end{enumerate}
	In all these cases, $u$ is a substring of $(F_{n-1})^R$,
	which means $u$ has an occurrence that crosses $p'$ or $q'$.
	Therefore, $\{ p', q' \}$ is an attractor of $F_n$.
\end{proof}

\begin{figure}[t]
	\centering
	\includegraphics[width=\linewidth]{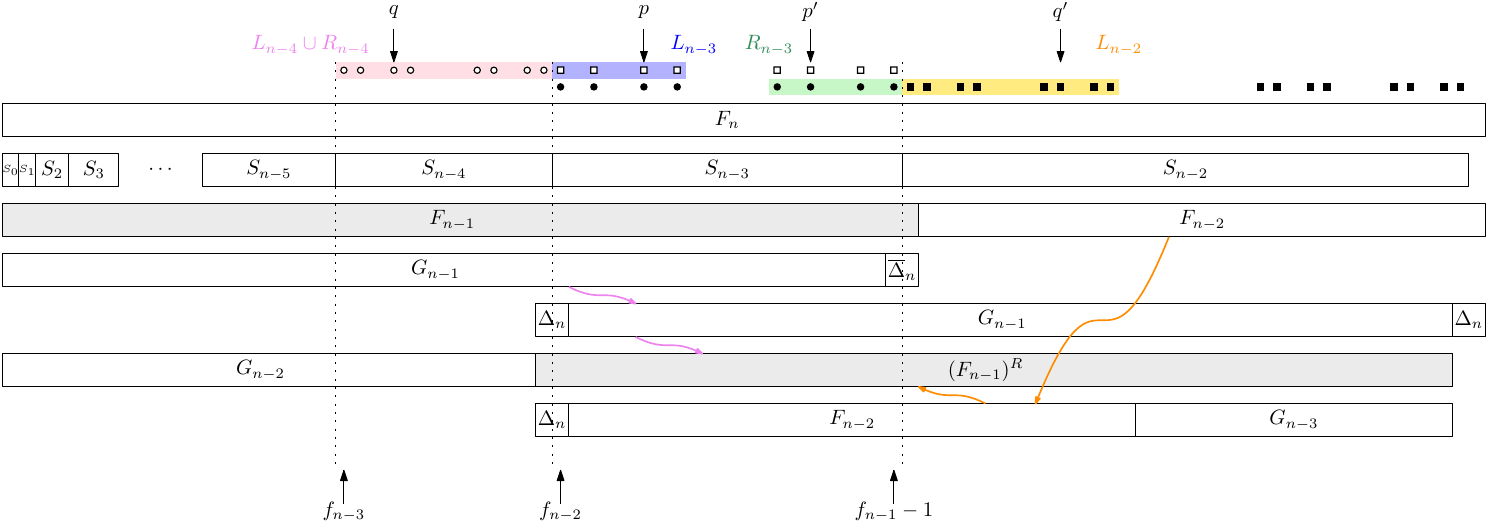}
	\caption{
		Illustration of \Cref{lem:valid_att_from_flip} for $n=10$.
		White and black markers at the top denote positions in $\attof{F_{n-1}}$ and $\attof{F_n}$, respectively.
		Highlighted markers, together with the highlighted occurrences of $F_{n-1}$ and $(F_{n-1})^R$, illustrate the mapping described in the proof.
		Pink and orange squiggly arrows illustrate Case~\ref{flip-case:in-g-n-1} and Case~\ref{flip-case:in-f-n-2}, respectively.
		(Note that $\Delta_{n-1} = \overline{\Delta_n}$ and $(\Delta_{n-1})^R=\Delta_n$.)
	}
	\label{fig:valid-att-from-flip}
\end{figure}

\begin{lemma}\label{lem:valid_att_from_flip_offset}
	For $n \geq 7$,
	if $R_{n-3}\otimes L_{n-2}\subseteq \attof{F_n}$, then $L_{n-3}\otimes L_{n-2} \subseteq \attof{F_n}$.
\end{lemma}

\begin{proof} See~\Cref{fig:valid-att-from-flip-offset}.
	For any $\{p, q\} \in R_{n-3}\otimes L_{n-2}\subseteq \attof{F_n}$ with $p < q$,
	let $p' = p - f_{n-4}$.
	We have $\{ p' , q \} \in L_{n-3}\otimes L_{n-2}$ since $R_{n-3} = L_{n-3}\oplus f_{n-4}$ (\Cref{eqn:r_n}).
	Since the mapping $\{ p,q \}\mapsto \{p', q\}$ is bijective, it suffices to show that $\{ p', q \}\in \attof{F_n}$.

	Since $\{ p, q\}$ is an attractor of $F_n$, all substrings of $F_n$ have an occurrence crossing $p$ or $q$.
	Consider substrings of $F_n$ that have occurrences crossing $p$ but not $q$ nor $p'$.
	Since
	$\min L_{n-3} = f_{n-2} \leq p' < p < q \leq \max L_{n-2} = 2f_{n-2}-1$
	and $F_n = \underbrace{\overbrace{F_{n-2}}^{f_{n-2}}G_{n-2}}_{2f_{n-2}-2}\Delta_n G_{n-3}\Delta_n$,
	all such occurrences are in a substring of $G_{n-2}$.
	Furthermore, they have an occurrence crossing $p' = p-f_{n-4}$, since
	$F_n = F_{n-1}F_{n-2} = F_{n-3}G_{n-2}\Delta_{n-1}F_{n-2}$
	and $G_{n-2}$ also occurs $f_{n-4}=f_{n-2}-f_{n-3}$ positions to the left.
	Thus, all substrings of $F_n$ have an occurrence crossing $p'$ or $q$,
	and therefore, $\{ p', q \}$ is an attractor of $F_n$.
\end{proof}

\begin{figure}[t]
	\centering
	\includegraphics[width=\linewidth]{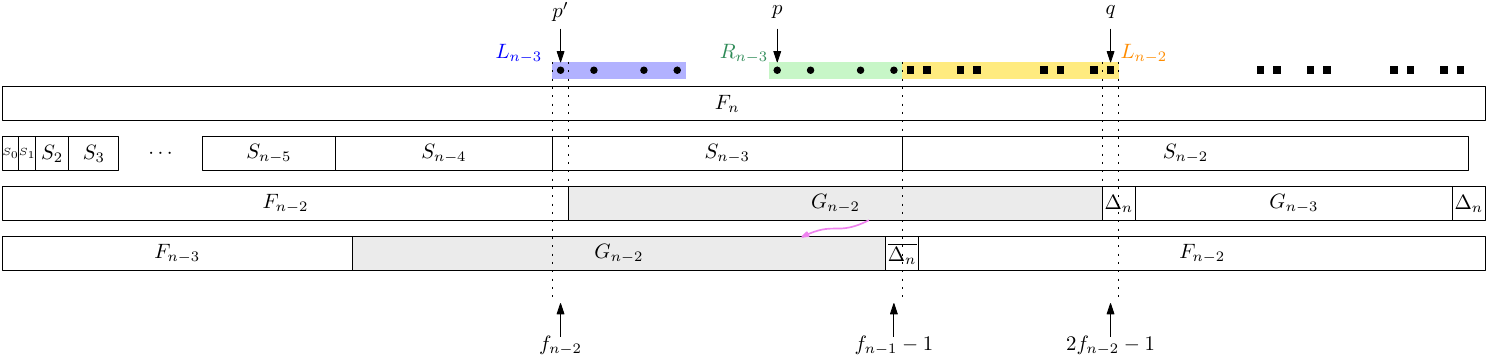}
	\caption{
		Illustration of \Cref{lem:valid_att_from_flip_offset} for $n=10$.
		Black markers at the top denote positions in $\attof{F_n}$.
		Relevant attractor positions and occurrences of $G_{n-2}$ are highlighted.
	}
	\label{fig:valid-att-from-flip-offset}
\end{figure}

\begin{lemma}\label{lem:valid_att_from_n-2_offset}
	For $n \geq 9$,
	if $L'_{n-5}\otimes (L_{n-4}\cup R_{n-4})\subseteq \attof{F_{n-2}}$,
	then $L'_{n-3}\otimes R_{n-2}\subseteq \attof{F_n}$.
\end{lemma}

\begin{proof} See~\Cref{fig:valid-att-from-n-2}.
	For any $\{ p, q \} \in L'_{n-5}\otimes (L_{n-4}\cup R_{n-4})\subseteq \attof{F_{n-2}}$ with
	$p\in L'_{n-5}$ and $q\in L_{n-4}\cup R_{n-4}$,
	let $p' = p + f_{n-1}$, $q' = q + f_{n-1}$,
	and $p'' = p + f_{n-3} = p' - f_{n-2}$.
	Notice that
	$f_{n-2} \leq p'' \leq f_{n-2}+1$,
	$p' \leq 2f_{n-2}+1$, and $q' \leq f_n-1$ can be seen
	from~\Cref{eqn:l_n_prime} and~\Cref{obs:l_r_min_max}.
	It can be also seen that $p'' \in L'_{n-5}\oplus f_{n-3} = L'_{n-3}$ (\Cref{eqn:l_n_prime_recursion}),
	and $q' \in (L_{n-4}\cup R_{n-4})\oplus f_{n-1}=R_{n-2}$ (\Cref{eqn:r_n}).
	Thus, $\{p'', q'\} \in L'_{n-3}\otimes R_{n-2}$.
	Since the mapping $\{p, q\} \mapsto \{ p'', q'\}$ is bijective,
	it suffices to show that $\{p'',q'\}\in\attof{F_n}$.

	Consider any substring $u = F_n[i .. j]$ that does not cross $p''$ nor $q'$ and the following cases.
	\begin{enumerate}
		\item If $j < p''$, then $u$ is a substring of $F_{n-2}$ since $p'' \leq f_{n-2}+1$.
		\item If $p'' < i \leq j < p'$, then $u$ is a substring of $F_{n-2}$ since
		      $f_{n-2}\leq p'' < p' \leq 2f_{n-2}+1$, and $(F_{n-2})^2$ is a prefix of
		      $F_n = F_{n-1}F_{n-2}=F_{n-2}F_{n-3} F_{n-4}F_{n-5}F_{n-4} = (F_{n-2})^2F_{n-5}F_{n-4}$.
		\item\label{case:crossing_p_prime}
		      If $p'' < i \leq p' \leq j < q'$, then $u$ is a substring of $G_{n-1}$ which crosses $p'$,
		      since $f_{n-2} \leq p'' < q' \leq f_{n}-1$ and $F_n = F_{n-2}G_{n-1}\Delta_{n}$.
		\item If $p' < i$, then $u$ is a substring of $F_{n-2}$, since $F_n = F_{n-1}F_{n-2}$ and $p' \geq 2f_{n-2} > f_{n-1}$.
	\end{enumerate}
	For Case~\ref{case:crossing_p_prime}, $u$ has an occurrence that crosses $p''$, since $p'' = p' - f_{n-2}$,
	and there is an occurrence of $G_{n-1}$, $f_{n-2}$ positions to the left (as a prefix of $F_n$).
	It remains to show that all substrings of $F_{n-2}$ have an occurrence crossing $p''$ or $q'$.
	Since $\{ p, q \}$ is an attractor of $F_{n-2}$,
	$F_n = F_{n-1}F_{n-2}=F_{n-2} F_{n-3} F_{n-2}$, and $\{p',q'\} = \{p,q\}\oplus f_{n-1}$,
	all substrings of $F_{n-2}$ have an occurrence that crosses $p'$ or $q'$.
	Furthermore, as was the case for Case~\ref{case:crossing_p_prime},
	any substring of $F_{n-2}$ crossing $p'$ but not crossing $q'$ is a substring of $G_{n-1}$,
	and has an occurrence crossing $p''$.
	Therefore, any substring will cross $p''$ or $q'$, and $\{p'', q'\}$ is an attractor of $F_n$.
\end{proof}

\begin{figure}[t]
	\centering
	\includegraphics[width=\linewidth]{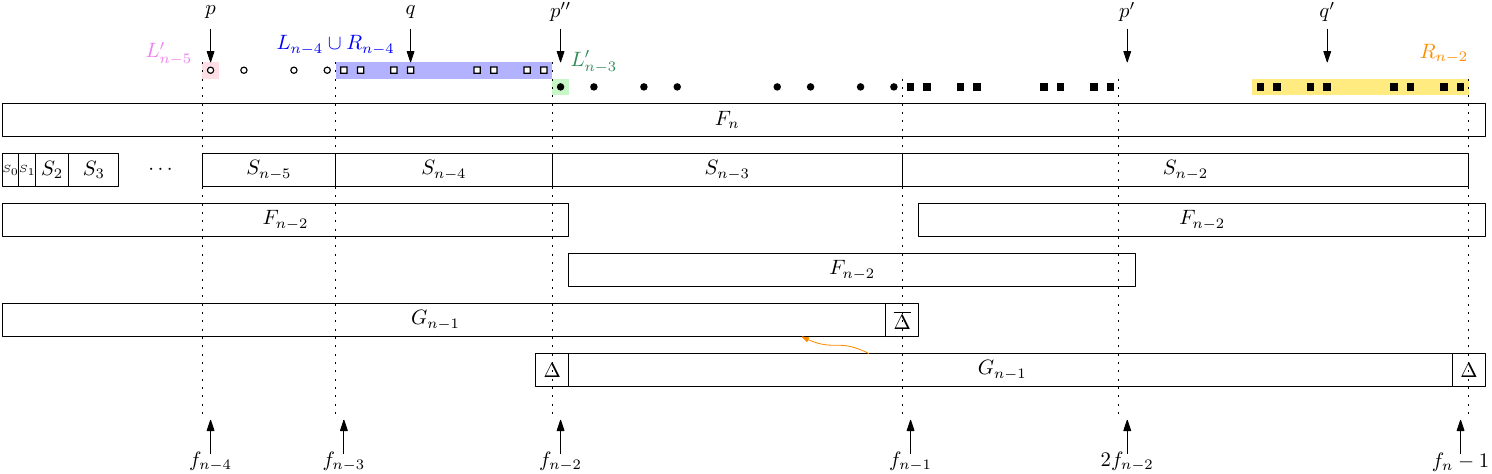}
	\caption{
		Illustration of \Cref{lem:valid_att_from_n-2_offset} for $n=10$.
		White and black markers at the top denote positions in $\attof{F_{n-2}}$ and $\attof{F_n}$, respectively.
	}
	\label{fig:valid-att-from-n-2}
\end{figure}

\begin{proof}[Proof of~\Cref{thm:all_smallest_attractors_of_fib}]
	The proof is by induction.
	$\attof{F_7}$ and $\attof{F_8}$ can be confirmed by brute force computation:
	\begin{align*}
		\attof{F_7} & = ((L_{4}\cup R_{4})\otimes L_{5})\cup(L'_{4}\otimes R_{5})                              \\
		            & = ((\{8,9 \}\cup\{ 11,12\})\otimes \{13,15\})\cup(\{8,9\}\otimes \{18,20\})              \\
		            & = \{
		\{ 8, 13 \}, \{ 8, 15\}, \{ 8,18\}, \{ 8,20\},%
		\{ 9, 13 \}, \{ 9, 15\}, \{ 9,18\}, \{ 9,20\},                                                         \\
		            & \phantom{= \{,} \{ 11, 13 \}, \{ 11, 15\},                                               %
		\{ 12, 13 \}, \{ 12, 15\}\}, \text{ and}                                                               \\
		\attof{F_8} & = ((L_{5}\cup R_{5})\otimes L_{6})\cup(L'_{5}\otimes R_{6})                              \\
		            & = ((\{ 13, 15\}\cup \{ 18,20\})\otimes\{21,22,24,25\})\cup(\{13\}\otimes\{29,30,32,33\}) \\
		            & = \{
		\{13,21\},\{13,22\},\{13,24\},\{13,25\},\{13,29\},\{13,30\},\{13,32\},\{13,33\},                       \\
		            & \phantom{= \{,} \{15,21\},\{15,22\},\{15,24\},\{15,25\},                                 %
		\{18,21\},\{18,22\},\{18,24\},\{18,25\},                                                               \\
		            & \phantom{= \{,} \{20,21\},\{20,22\},\{20,24\},\{20,25\}
		\}.
	\end{align*}
	Under the induction hypothesis and~\Cref{note:Lprime_range}, the conditions of~\Cref{lem:valid_att_from_flip,lem:valid_att_from_flip_offset,lem:valid_att_from_n-2_offset}
	hold for $n\geq 9$, and thus, together with~\Cref{cor:attoffib_contained_in_formula}
	and the above base cases, the theorem holds.
\end{proof}

\begin{corollary}
	For $k \geq 4$,
	$|\attof{F_{2k-1}}| = (2^{k-3} + 1)2^{k-2}$
	and
	$|\attof{F_{2k}}|   = (2^{k-2} + 1)2^{k-2}$.
\end{corollary}

\begin{proof}
	It is clear that $|L_2|=|R_2|=|L_3|=|R_3|=1$,
	and $|L_i| = |R_i| = 2|L_{i-2}|=2|R_{i-2}|$ for $i \geq 4$.
	Therefore,
	$|L_{2k-1}| = |R_{2k-1}| = 2^{k-2}$ and
	$|L_{2k}| = |R_{2k}| = 2^{k-1}$.
	Since $|L'_{2k-1}|=1$ and $|L'_{2k}|=2$,
	we have from~\Cref{thm:all_smallest_attractors_of_fib} that
	$|\attof{F_{2k-1}}| = |((L_{2k-4}\cup R_{2k-4})\otimes L_{2k-3})\cup(L'_{2k-4}\otimes R_{2k-3})|
		= (2\cdot 2^{k-3}\cdot 2^{k-3}) + 2 \cdot 2^{k-3} = (2^{k-3}+1)2^{k-2}$ and
	$|\attof{F_{2k}}| = |((L_{2k-3}\cup R_{2k-3})\otimes L_{2k-2})\cup(L'_{2k-3}\otimes R_{2k-2})|
		= (2\cdot 2^{k-3}\cdot 2^{k-2}) + 2^{k-2} = (2^{k-2}+1)2^{k-2}$.
\end{proof}

\section{The Smallest Attractors of Period-Doubling Words}
\begin{definition}[$D_n$]
	For each $n\geq 0$, the $n$th period-doubling word is defined as
	$D_0 = \texttt{a}$, and $D_n = \phi(D_{n-1})$ where $\phi$ is a morphism defined by
	$\phi(\texttt{a})=\texttt{ab}$ and $\phi(\texttt{b})=\texttt{aa}$. Note that $|D_n| = 2^n$.
\end{definition}
Let $D_\infty$ denote the infinite period-doubling word, of which $D_n$ is a prefix for every $n\geq 0$.

\begin{theorem}\label{thm:pd_attractors}
	$\attof{D_2} = \{ \{ 2, 3 \}, \{ 2, 4\} \}$.
	For $n \geq 3$,
	$\attof{D_n} = \{ \{ p_n, r_n \},$ $\{ q_n, r_n \} \}$,
	where
	$p_n = 3\cdot 2^{n-3}$, $q_n = 2^{n-1}$, and $r_n = 3\cdot 2^{n-2}$.
\end{theorem}
Schaeffer and Shallit~\cite{schaeffer2024stringattractorsautomaticsequences} showed that for arbitrary prefixes of length at least $6$ of the infinite period-doubling sequence,
$\{ p_n, r_n \}$ is a smallest attractor when the length is in $[2^n, 3\cdot 2^{n-1})$,
and $\{ q_{n}, q_{n+1}\}$, when the length is in $[3\cdot 2^{n-1},2^{n+1})$.
We consider only the length $2^n$ prefixes,
but characterize all smallest string attractors.
\begin{observation}\label{obs:pd_parity_b}
	Since any occurrence of $\texttt{b}$ in $D_n$ must have come from an occurrence of $\texttt{a}$ in $D_{n-1}$ by $\phi(\texttt{a})=\texttt{ab}$,
	an occurrence of $\texttt{b}$ in $D_n$ is always at an even position.
	Hence, $\texttt{bb}$ cannot occur.
	This implies that an occurrence of $\texttt{aa}$ starting at an odd position
	must be followed (provided it is not the end of the string) and preceded
	by $\texttt{ab}$,
	since it could only have come from $\phi(\texttt{aba})$.
\end{observation}
Furthermore, the above implies the following corollary.
\begin{corollary}\label{cor:pd_substring_parity}
	Any occurrence of a substring $x$ in $D_n$ with $|x| \geq 3$
	has the same parity, i.e.,
	$|\{ i \bmod 2 \mid D_n[i..i+|x|) = x\}| = 1$.
\end{corollary}
\begin{observation}\label{obs:pd_properties}
	The following are known or easily verifiable properties of period-doubling words.
	\begin{enumerate}[(1)]
		\item\label{obs:prefix_prime_in_square}
		      For $n \geq 1$, $(D_{n-1})^2 = D'_{n}\bitcomp{c}$ where $D'_n = D_n[1..2^n)$ and $c = D_{n}[|D_{n}|]$.
		\item\label{obs:pd_factorizations}
		      For $n \geq 2$, $D_n = D_{n-1}(D_{n-2})^2 = D'_{n-1}cD'_{n-1}\bitcomp{c}$, where $c = D_{n-1}[|D_{n-1}|]$.
		\item\label{obs:cube_in_square} For any $i$ such that $0 \leq n-2i < n$, $(D_{n-2i})^2$ is a suffix of $D_{n}$, and for any $0 \leq i < n$, $D_{i}$ is a prefix of $D_n$.
	\end{enumerate}
\end{observation}

We first show that the two position pairs are indeed attractors of $D_n$.
\begin{lemma}\label{lem:pd_valid_attractors}
	For $n \geq 3$,
	$\{ \{ p_n, r_n \}, \{ q_n, r_n \} \}\subseteq \attof{D_n}$
\end{lemma}

\begin{figure}[t]
	\centering
	\includegraphics[width=0.8\linewidth]{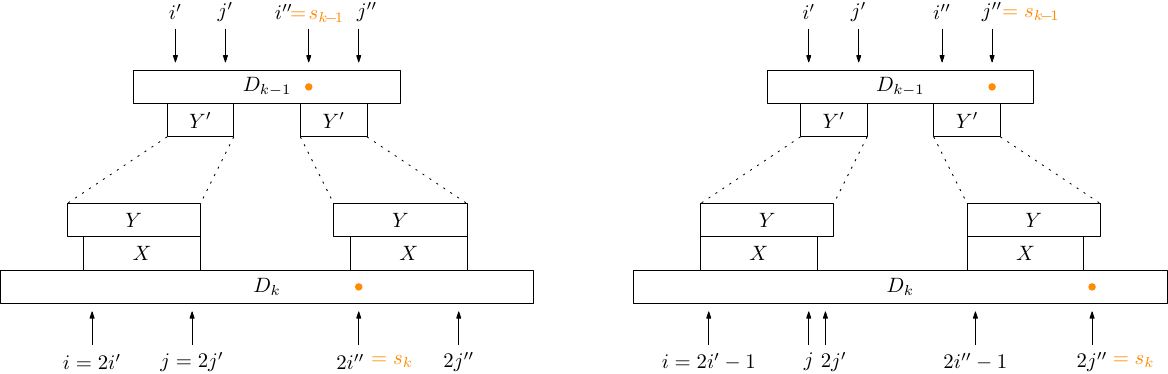}
	\caption{
		Illustration of the first paragraph of the proof of \Cref{lem:pd_valid_attractors}.
		Left: $X = Y[2..|Y|]$ and $i'' = s_{k-1}$.
		Right: $X = Y[1..|Y|-1]$ and $j'' = s_{k-1}$.
		(If $X = Y$, $X = Y[2..|Y|]$, or $X = Y[1..|Y|-1]$ and $j'' > s_{k-1}$,
		then $X$ has an occurrence crossing $s_k$.)
	}
	\label{fig:pd-x-y}
\end{figure}

\begin{figure}[t]
	\centering
	\includegraphics[width=0.8\linewidth]{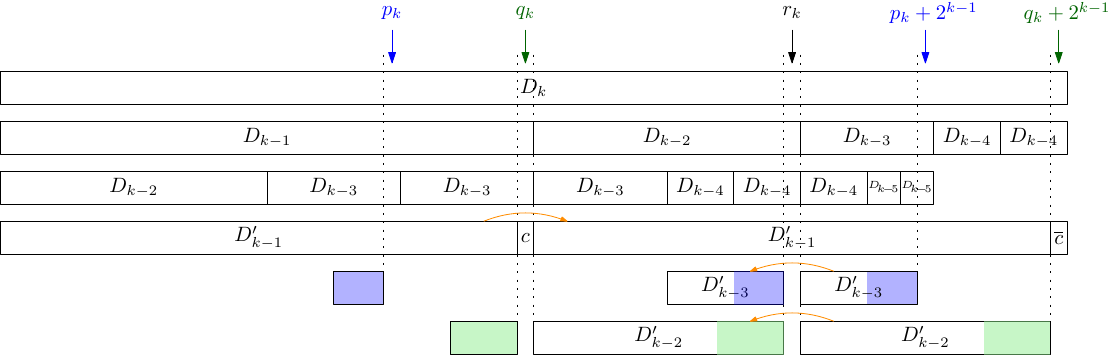}
	\caption{
		Illustration of the second paragraph of the proof of \Cref{lem:pd_valid_attractors}.
		Blue and green blocks represent substrings that have an occurrence ending just before $p_k$ and $q_k$, respectively, but do not have an occurrence crossing $r_k$ when considering its occurrence in the second half of the string via $D'_{k-1}$.
		Such substrings have occurrences ending just before $r_k$.
	}
	\label{fig:pd-x-p-q-r}
\end{figure}

\begin{proof}
	The statement can be verified by brute force for $n = 3$ and $4$.
	Suppose the statement holds for all $3\leq n < k$.
	Let $X = D_k[i..j]~(1\leq i\leq j\leq 2^k)$ be an arbitrary substring of $D_k$,
	and let $Y$ be the shortest even-length substring of $D_k$
	that contains the occurrence of $X$, starting at an odd position, and ending at an even position,
	i.e.,  $Y = D_k[2i'-1..2j']$ where $i' = \lceil \frac{i}{2}\rceil$, $j' = \lceil\frac{j}{2}\rceil$.
	It holds that $Y = \phi(Y')$ where $Y'=D_{k-1}[i'..j']$.
	Since $\{p_{k-1}, r_{k-1}\}$ (resp.\ $\{q_{k-1},r_{k-1}\}$) are attractors of $D_{k-1}$,
	$Y'$ has an occurrence $D_{k-1}[i''..j'']=Y'$ crossing
	$p_{k-1}$ (resp.\ $q_{k-1}$) or $r_{k-1}$, i.e.,
	$i'' \leq s_{k-1} \leq j''$ for some $s_{k-1} \in \{ p_{k-1}, q_{k-1}, r_{k-1} \}$.
	Then, since $2i''-1 < 2s_{k-1}=s_{k} \leq 2j''$,
	$Y[2..|Y|]$ in the corresponding occurrence of $Y = D_k[2i''-1..2j'']$ will cross $s_k$.
	Since $Y[2..|Y|-1]$ is a substring of $X$,
	$X$ has an occurrence crossing $s_k$ or ending just before $s_k$.
	See~\Cref{fig:pd-x-y}.
	Thus, to prove that $\{ p_k, r_k \}$ and $\{ q_k, r_k \}$ are attractors of $D_k$,
	it remains to show that any substring $X$ ending just before
	$p_k$ (resp.\ $q_k$) or $r_k$ will have an occurrence crossing $p_k$ (resp.\ $q_k$) or $r_k$.

	Since $p_k = |D_{k-2}D_{k-3}|$, $q_k = |D_{k-1}|$,
	and $D_k = D'_{k-1}cD'_{k-1}\bitcomp{c}$ (\Cref{obs:pd_properties}~(\ref{obs:pd_factorizations})),
	any substring $X$ ending just before $p_k$ (resp.\ $q_k$)
	has an occurrence ending just before $p_k+2^{k-1} = |D_{k-1}D_{k-2}D_{k-3}|$
	(resp.\ $q_k+2^{k-1} = |D_k|$).
	Since $D_k = D_{k-1}(D_{k-2})^2 = D_{k-1}D_{k-2}D_{k-3}(D_{k-4})^2$ and
	$r_k = |D_{k-1}D_{k-2}|$,
	any such $X$ not crossing $r_k$ must be a suffix of
	$D'_{k-3}$ (resp.\ $D'_{k-2}$),
	having an occurrence ending just before~$r_k$.
	See~\Cref{fig:pd-x-p-q-r}.

	Finally, we claim that any substring $X$ ending just before $r_k$
	will have occurrences crossing both $p_k$ and $q_k$, or $r_k$.
	We consider the following $O(k)$ cases (which together cover all lengths) depending on $|X|$.
	\begin{enumerate} %
		\item $|X| > |D_{k-3}D'_{k-2}|$. See~\Cref{fig:pd-proof}~(a). Since the starting position of $X$
		      is $r_k - |X| < r_k - |D_{k-3}D'_{k-2}| = 3\cdot 2^{k-2}-2^{k-3}-2^{k-2}+1 \leq p_k+1$, $X$ crosses both $p_k$ and $q_k$.
		\item\label{case:pd_proof_2} $|D_{k-3}D'_{k-2}| \geq |X| \geq |(D_{k-4})^2|$.
		      See~\Cref{fig:pd-proof}~(a) and~\Cref{fig:pd-interval-case-2}.
		      Since $D_{k-3}D'_{k-2} = D_k[p_k+1..r_k-1] = D_k[p_k+1-|D_{k-3}|..r_k-1-|D_{k-3}|]$,
		      it follows that $D_{k-3}D_{k-3}D'_{k-2}$ has a period of $|D_{k-3}|$.
		      Since $|X| \geq |(D_{k-4})^2| > |D'_{k-3}|$, $X$ starts with a suffix of $D_{k-3}$
		      and therefore will have occurrences crossing $p_k$ and $q_k$. %
		      More precisely, consider the occurrence $D_k[r_k-|X|..r_k-1]$ of $X$, which lies in $D_k[2^{k-2}+1..r_k-1] = D_{k-3}D_{k-3}D'_{k-2}$.
		      Shifting it to the left by $|D_{k-3}| = 2^{k-3}$ gives an occurrence crossing $q_k$, which also crosses $p_k$ if $|X|\geq 2^{k-2}$;
		      if $|X| < 2^{k-2}$, shifting it to the left by $2^{k-2}$ gives an occurrence crossing $p_k$.
		      In both cases, the shifted occurrences remain within $D_k[2^{k-2}+1..r_k-1]$, which has period $2^{k-3}$.
		\item\label{case:pd_proof_3} $|(D_{k-2i})^2| > |X| \geq |D_{k-2i}| = |(D_{k-2i-1})^2|$ for $i \geq 2, k-2i-1\geq 0$.
		      See~\Cref{fig:pd-proof}~(b) and~\Cref{fig:pd-cases-3-4}.
		      Due to~\Cref{obs:pd_properties}~(\ref{obs:cube_in_square}),
		      there is an occurrence of $(D_{k-2i})^2$ centered at (i.e., the first copy ending at) both
		      $r' = r_k-|(D_{k-2i})|$ and $r_k$.
		      Since $X$ is a suffix of $D_{k-2i}D'_{k-2i}$ and crosses~$r'$, it also has an occurrence crossing $r_k$.
		\item\label{case:pd_proof_4} $|(D_{k-2i-1})^2| > |D'_{k-2i}| \geq |X| \geq |D_{k-2i-1}|=|(D_{k-2i-2})^2|$ for $i \geq 2, k-2i-2\geq 0$.
		      See~\Cref{fig:pd-proof}~(b) and (c), and~\Cref{fig:pd-cases-3-4}.
		      $X$ is a suffix of $D'_{k-2i}$.
		      There is an occurrence of $(D_{k-3})^2$ centered at both $p_k$ and $q_k$.
		      Due to~\Cref{obs:pd_properties}~(\ref{obs:cube_in_square}),
		      there must be an occurrence of $(D_{k-2i-1})^2$ also centered at $p_k$ and $q_k$,
		      and thus an occurrence of $D'_{k-2i}$ (\Cref{obs:pd_properties}~(\ref{obs:prefix_prime_in_square}))
		      of which $X$ is a suffix.
		      Since $|X| \geq |D_{k-2i-1}| > |D'_{k-2i-1}|$,
		      $X$ has occurrences that cross $p_k$ as well as $q_k$.
		\item $|X| = 1$. $D_3[p_3] = D_3[q_3] = \texttt{a}$ and $D_3[r_3]=\texttt{b}$.
		      It is easy to see that $D_k[p_k] = D_k[q_k] = c$ and $D_k[r_k]=\bitcomp{c}$ from the definition of $\phi$, so $X$ will cross both $p_k$ and $q_k$, or $r_k$.%
	\end{enumerate}
	Thus, $\{\{p_k,r_k\},\{q_k,r_k\}\}\subseteq \attof{D_k}$, proving the lemma.
\end{proof}

\begin{figure}[t]
	\centering
	\includegraphics[width=0.85\textwidth]{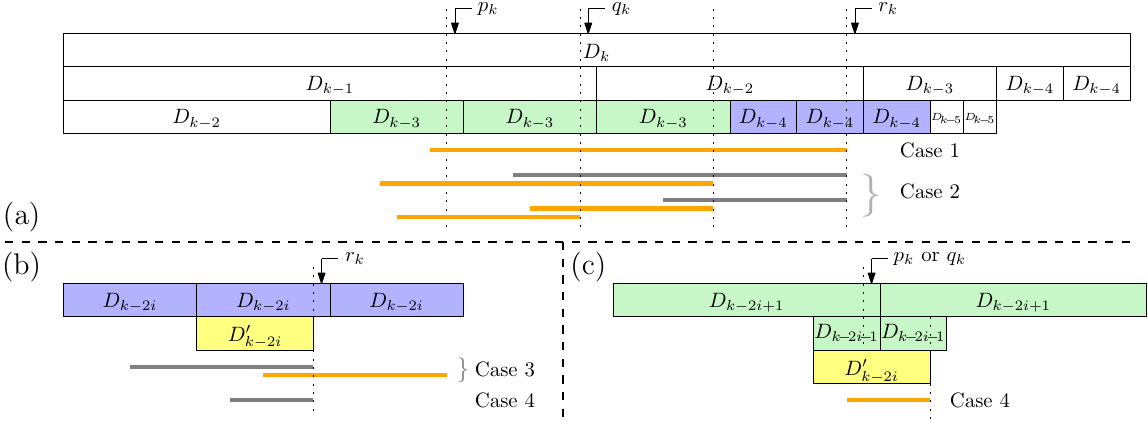}
	\caption{Illustration of the proof of~\Cref{lem:pd_valid_attractors}.
		There are 4 main cases depending on the length of $X$
		which is a suffix of $D_k[1..r_k)$ and shown in gray (except for Case 1, which already crosses both $p_k$ and $q_k$),
		and the occurrence crossing both $p_k$ and $q_k$, or $r_k$ is shown in orange.
	}\label{fig:pd-proof}
\end{figure}

\begin{figure}[t]
	\centering
	\includegraphics[width=0.75\linewidth]{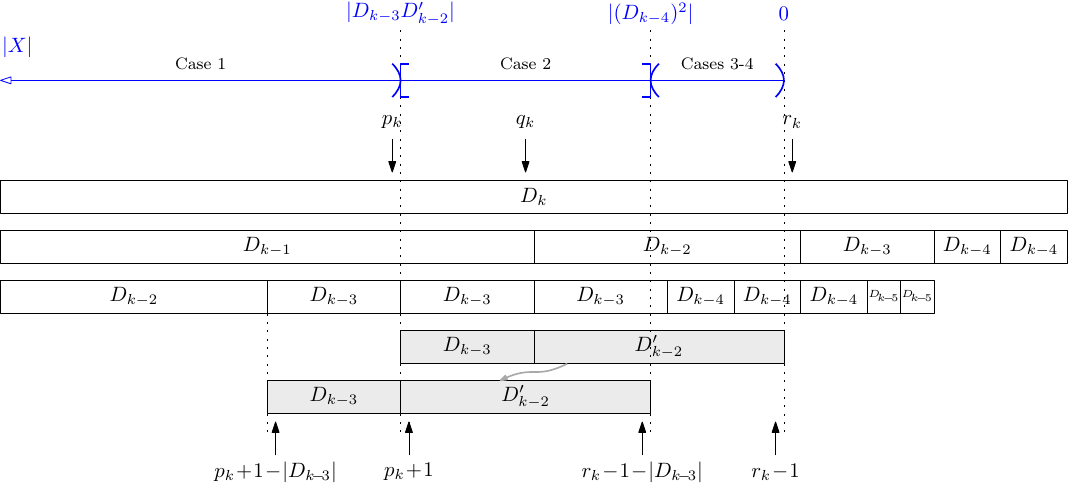}
	\caption{
		Illustration of Case~\ref{case:pd_proof_2} in the proof of \Cref{lem:pd_valid_attractors}.
	}
	\label{fig:pd-interval-case-2}
\end{figure}

\begin{figure}[t]
	\centering
	\includegraphics[width=\linewidth]{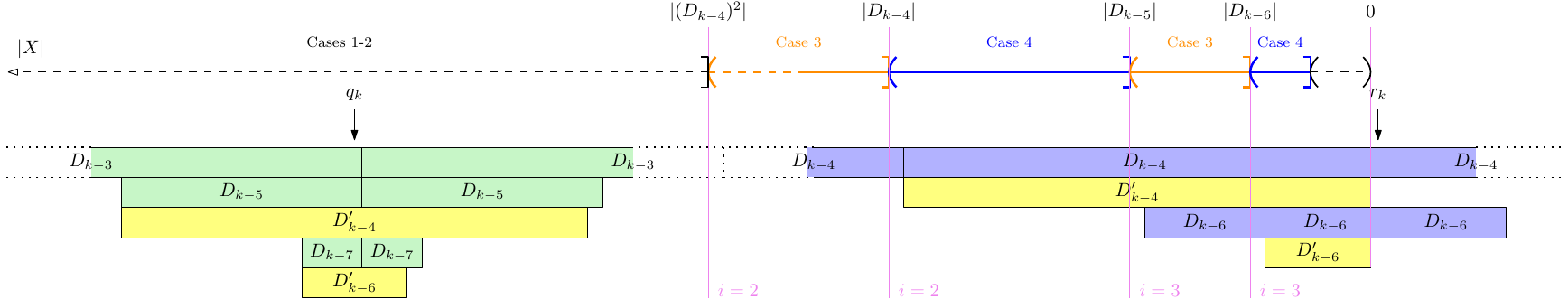}
	\caption{
		Illustration of Cases~\ref{case:pd_proof_3}--\ref{case:pd_proof_4} in the proof of \Cref{lem:pd_valid_attractors} for $i=2$ and $i=3$.
		The highlighted substrings follow the same color pattern in (b) and (c) of \Cref{fig:pd-proof}.
	}
	\label{fig:pd-cases-3-4}
\end{figure}

Next,
we will show that no other position pairs can be attractors of $D_n$ with the help of the following lemma.
\begin{lemma}\label{lem:pd_attractors_completeness_1}
	For $n \geq 4$, let $\{ p, q \} \in \attof{D_n}$ and
	$D_n[p] = \texttt{a}$ and $D_n[q]=\texttt{b}$.
	Then,
	$p,q$ are both even and $\{p',q'\} = \{ p/2, q/2 \} \in \attof{D_{n-1}}$.
\end{lemma}
\begin{proof}
	The strings $\texttt{aaa}$, $\texttt{aab}$, $\texttt{baa}$, $\texttt{bab}$, $\texttt{abab}$, $\texttt{aabaa}$, $\texttt{aabab}$, $\texttt{babab}$, $\texttt{babaa}$, $\texttt{ababab}$ are all substrings of $D_n$ for $n\geq 4$.
	It is clear that $q$ is even from~\Cref{obs:pd_parity_b}.

	We first claim that if $p<2^n$, then $D_n[p+1] \neq \texttt{b}$.
	Suppose to the contrary that $D_n[p+1] = \texttt{b}$.
	Then, for $\texttt{aaa}$ to have an occurrence crossing $p$,
	$D_n[p-2..p+1] = \texttt{aa\underline{a}b}$,
	where the underlined $\texttt{a}$ corresponds to position $p$.
	From~\Cref{obs:pd_parity_b}, we have $D_n[q-1..q]=\texttt{a\underline{b}}$,
	where the underlined
	$\texttt{b}$ corresponds to position $q$.
	In order for the substrings $\texttt{baa}$ and $\texttt{bab}$ to have an occurrence crossing one of these positions,
	$D_n[q-3..q+4]=\texttt{\textbf{a}ba\underline{b}aa\textbf{ab}}$, where the bold
	$\texttt{\textbf{a}}$ and $\texttt{\textbf{ab}}$ are due to~\Cref{obs:pd_parity_b}.
	Next, in order for the substrings $\texttt{aabaa}$ and $\texttt{aabab}$ to have an occurrence crossing $p$ or $q$,
	$D_n[p-2..p+3] = \texttt{aa\underline{a}baa}$ and $D_n[q-5..q+4] = \texttt{\textbf{a}aaba\underline{b}aaab}$.
	However, then, there cannot be an occurrence of $\texttt{ababab}$ that crosses $p$ or $q$ contradicting that they are an attractor.

	Next, we claim that $p$ is even. If it is odd, then, from the above arguments,
	$D_n[p..p+3] = \texttt{\underline{a}aab}$.
	We also have $D_n[q-1..q]=\texttt{a\underline{b}}$.
	For $\texttt{aab}$ and $\texttt{abab}$ to cross $p$ or $q$,
	$D_n[q-3..q+2]=\texttt{\textbf{a}aa\underline{b}ab}$. Furthermore, for $\texttt{aabaa}$ to cross $p$ or $q$,
	$D_n[p-3..p+3] = \texttt{aab\underline{a}aab}$.
	However, then at least one of $\texttt{babab}$ and $\texttt{babaa}$ cannot cross $p$ or $q$.

	Finally, consider any substring $X'$ of $D_{n-1}$.
	If $|X'|=1$, then, from the above arguments,
	$D_{n-1}[p'] = \texttt{b}$ since $D_{n}[p]=\texttt{a}$ could not have been derived by $\phi(\texttt{a})$,
	and $D_{n-1}[q'] = \texttt{a}$ since $D_{n}[q]=\texttt{b}$ must have been derived by $\phi(\texttt{a})$.
	Hence, $X'$ crosses $p'$ or $q'$.
	If $|X'|\geq 2$, consider the corresponding substring $X = \phi(X')$ in $D_n$.
	Since $\{p, q\}$ is an attractor of $D_n$, $X$ has an occurrence crossing $p$ or $q$.
	Furthermore, since $|X|\geq 4$ and the parity of its occurrences is determined uniquely~(\Cref{cor:pd_substring_parity}),
	an occurrence of $X$ at position $k$ in $D_n$ crossing $p$ or $q$
	implies an occurrence of $X'$ at position
	$\lceil k/2\rceil$ in $D_{n-1}$ crossing $p'$ or $q'$.
\end{proof}

\begin{proof}[Proof of~\Cref{thm:pd_attractors}]
	The cases $n=2$ and $n=3$ can be verified by brute force computation.
	For $k\geq 4$, assume the statement of~\Cref{thm:pd_attractors}
	holds for all $3\leq n < k$.
	\Cref{lem:pd_attractors_completeness_1}
	implies that any element other than $\{p_k,r_k\}$ or $\{q_k,r_k\}$ in $\attof{D_k}$
	would imply an element in $\attof{D_{k-1}}$ not in $\{\{p_{k-1},r_{k-1}\},\{q_{k-1},r_{k-1}\}\}$
	contradicting the induction hypothesis.
	Together with~\Cref{lem:pd_valid_attractors}, this completes the proof.
\end{proof}

\section{The Smallest Attractors of Thue--Morse Words}

\begin{definition}[$\tm_n$]
	For each $n \geq 0$, the $n$th \emph{Thue--Morse word} is defined as $\tm_0 = \a$,
	and $\tm_n = \morphtm(\tm_{n-1})$ where $\morphtm$ is a morphism defined by $\morphtm(\a) = \a\b$ and $\morphtm(\b) = \b\a$.
	Note that $|\tm_n| = 2^n$.
\end{definition}
Let $\tm_\infty$ denote the infinite Thue--Morse word, of which $\tm_n$ is a prefix for every $n\geq 0$.

\begin{observation}[Lemma~4 in \cite{DBLP:conf/spire/BannaiFIKMN21}]\label{obs:tm_properties}
	For any substring $w \not\in \{ \a\b\a, \b\a\b, \a\b, \b\a, \a, \b \}$ of $\tm_n$,
	the parities of all occurrences of $w$ in $\tm_n$ are the same.
\end{observation}

\begin{lemma}[Corollary 4.1.1 in~\cite{DBLP:journals/dam/Brlek89}]\label{lem:tmunique_substring}
	Every substring of $\tm_n$ longer than $2^{n-2}$ is unique.
\end{lemma}

\begin{definition}[$B_{i, n}$, $\mathcal{B}_n$, $C_{i,n}$, $\mathcal{C}_n$]
	For $n \geq 4$ and $2 \leq i \leq 6$, define
	$u = i \cdot 2^{n-3}$,
	$B_{i,n} =  \{u , u + 1 \}$ and
	$\mathcal{B}_n = B_{2,n} \otimes B_{4,n} \otimes B_{6,n}$.
	Also, let $C_{i,n} = \{ u-1, u+2\} \cup B_{i,n}$ and
	$\mathcal{C}_n = C_{2,n} \otimes C_{4,n} \otimes C_{6,n}$.
\end{definition}

\begin{remark}\label{remark:tm-centers}
	See \cref{fig:tm-att-example}.
	In $\tm_n$ for $n \geq 4$,
	\begin{itemize}
		\item $B_{4,n}$ corresponds to the boundary between $\tm_{n-1}$ and $\bitcomp{\tm_{n-1}}$ in the only occurrence of $\tm_{n-1} \bitcomp{\tm_{n-1}} = \tm_n$;
		\item $B_{2,n}$ corresponds to the boundary between $\tm_{n-2}$ and $\bitcomp{\tm_{n-2}}$ in the only occurrence of $\tm_{n-2} \bitcomp{\tm_{n-2}} = \tm_{n-1}$;
		\item $B_{6,n}$ corresponds to the boundary between $\bitcomp{\tm_{n-2}}$ and $\tm_{n-2}$ in the only occurrence of $\bitcomp{\tm_{n-2}} \tm_{n-2} = \bitcomp{\tm_{n-1}}$;
		\item $B_{3,n}$ and $B_{5,n}$ correspond to the boundaries between $\bitcomp{\tm_{n-3}}$ and $\tm_{n-3}$ in the only two occurrences of $\bitcomp{\tm_{n-3}} \tm_{n-3} = \bitcomp{\tm_{n-2}}$.
	\end{itemize}
	The only occurrences of these substrings in $\tm_n$ have been shown in~\cite{DBLP:journals/jda/RadoszewskiR12,DBLP:conf/cpm/GuoK25}.
\end{remark}

\begin{lemma}[Variant of Theorem~2 in \cite{DBLP:conf/spire/KutsukakeMNIBT20}]\label{lem:an_attractor_of_tm}
	For $n \geq 4$,
	$( \mathcal{B}_n  \otimes (B_{3,n} \cup B_{5,n} )) \subseteq \attof{\tm_n}$.
\end{lemma}
\begin{proof}
	The size of the smallest string attractor of $\tm_n$ for $n\geq 4$ is known to be $4$~(Theorem 3 in~\cite{DBLP:conf/spire/KutsukakeMNIBT20}).
	The proof of Theorem~2 in~\cite{DBLP:conf/spire/KutsukakeMNIBT20} establishes that
	any substring of $\tm_n$ of length at least~$2$ will have an occurrence
	in at least one of
	$\tm_{n-1}\bitcomp{\tm_{n-1}}$,
	$\tm_{n-2}\bitcomp{\tm_{n-2}}$,
	$\bitcomp{\tm_{n-2}}\tm_{n-2}$, and
	$\bitcomp{\tm_{n-3}}\tm_{n-3}$,
	and crosses the boundary.
	Only $\{ \min B_{2,n}, \min B_{3,n}, \min B_{4,n}, \min B_{6,n}\}$ is chosen for an attractor in~\cite{DBLP:conf/spire/KutsukakeMNIBT20}.
	However, since substrings $\a\a$ and $\b\b$ must cross---and therefore constitute---two of the boundaries,
	any choice of selecting one position from each of the four boundaries will include positions crossed by $\a$ and $\b$,
	and thus will be an attractor for $\tm_n$.
	See \cref{fig:tm-att-example}.
\end{proof}

\begin{figure}[t]
	\centering
	\includegraphics[width=0.8\linewidth]{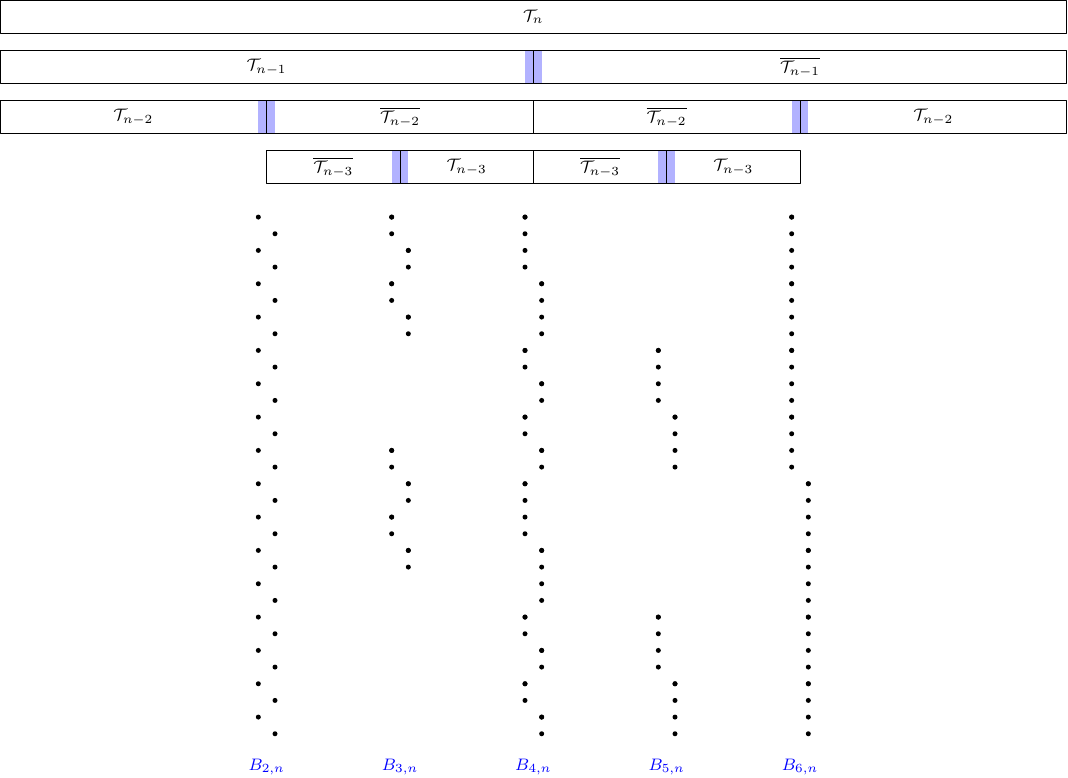}
	\caption{
		Illustration of $\attof{\tm_n}$ for $n \geq 6$.
		The blue boundaries illustrate \cref{remark:tm-centers} and the proof of \cref{lem:an_attractor_of_tm}.
	}
	\label{fig:tm-att-example}
\end{figure}

\begin{lemma}\label{lem:tm_attractors_completeness}
	For $n \geq 5$, if $\singleatt \in \attof{\tm_n}$,  then $\singleatt' = \{ \lceil p/2 \rceil : p \in \singleatt \} \in \attof{\tm_{n-1}}$.
\end{lemma}

\begin{proof}
	The strings $\a\b\a\b$ and $\b\a\b\a$ are substrings of $\tm_n$ for $n \geq 4$.

	We first claim that there exist positions $p'_\a,p'_\b \in A'$ such that $\tm_{n-1}[p'_\a] = \a$ and $\tm_{n-1}[p'_\b] = \b$.
	Since  $\tm_{n-1}[i] \in \{ \a, \b \}$ for any $1\leq i\leq 2^{n-1}$, it suffices to show that it cannot be that
	$\tm_{n-1}[p'] = \a$ for all $p'\in A'$, or $\tm_{n-1}[p'] = \b$ for all $p'\in A'$.
	Suppose to the contrary that $\tm_{n-1}[p'] = \a$ for all $p'\in A'$.
	Then, due to $\morphtm$,
	it holds that $\tm_n[2\lceil p/2\rceil-1..2\lceil p/2\rceil] = \underline{\a\b}$ for any $p\in A$.
	Furthermore, since $\tm_{n-1}$ does not contain $\a\a\a$ as a substring
	and again due to $\morphtm$,
	the region surrounding the positions in $\singleatt$ in $\tm_n$ can only be
	$\{
		\cdots\b\a\a\b\underline{\a\b}\b\a\cdots,
		\cdots\b\a\underline{\a\b}\a\b\b\a\cdots,
		\cdots\b\a\underline{\a\b}\b\a\cdots
		\}$ (or possibly a substring of them,
	when $p\in A$ is close to the start or end of $\tm_n$),
	where the underlined $\a\b$ corresponds to $\tm_n[2\lceil p/2\rceil-1..2\lceil p/2\rceil]$.
	However, then, there can be no occurrence of $\b\a\b\a$ crossing a position in $\singleatt$ contradicting that
	$\singleatt\in\attof{\tm_n}$.
	We can analogously rule out $\tm_{n-1}[p'] = \b$ for all $p'\in A'$
	by considering occurrences of $\a\b\a\b$ in
	$\{
		\cdots\a\b\b\a\underline{\b\a}\a\b\cdots,
		\cdots \a\b\underline{\b\a}\b\a\a\b \cdots,
		\cdots \a\b\underline{\b\a}\a\b \cdots
		\}$.

	Now, consider any substring $X'$ of $\tm_{n-1}$.
	If $|X'| = 1$, from the above arguments,
	there exist $p'_\a, p'_\b \in \singleatt'$ such that $X'$ crosses $p'_\a$ or $p'_\b$.
	If $|X'| \geq 2$, consider substring $X = \morphtm(X')$ of $\tm_n$,
	which has an occurrence crossing some position in $\singleatt$.
	Since $|X| \geq 4$, by \cref{obs:tm_properties}, the parities of the occurrences of $X$ are the same.
	Thus, an occurrence of $X$ at position $i$ in $\tm_n$ crossing a position in $\singleatt$ implies an occurrence of $X'$ at position $\lceil i/2 \rceil$ in $\tm_{n-1}$ crossing a position in $A'=\{ \lceil p/2 \rceil : p \in \singleatt \}$.
	Hence, $A'$ is an attractor of $\tm_{n-1}$.
	Since $n\geq 5$, we have $\gamma(\tm_{n-1})=\gamma(\tm_n)=4$~\cite{DBLP:conf/spire/KutsukakeMNIBT20},
	and $|A'|\leq |A|=4$.
	Therefore, $|A'|=4$ and $A'\in\attof{\tm_{n-1}}$.
\end{proof}

\begin{theorem}\label{thm:tm_smallest_attractors}
	For $n \geq 6$,
	$\attof{\tm_n} =  ( \mathcal{B}_n  \otimes (B_{3,n}  \cup B_{5,n} ))$.
\end{theorem}
To show~\Cref{thm:tm_smallest_attractors} we first show the following lemma.
\begin{lemma}\label{lem:tm_outer_positions}
	Let $k\geq 6$ and $A\in\attof{\tm_k}$. If $A\in\mathcal{C}_k\otimes(C_{3,k}\cup C_{5,k})$, then $A\in\mathcal{B}_k\otimes(B_{3,k}\cup B_{5,k})$.
\end{lemma}
\begin{proof}

	We show that there is no $p\in A$ such that $p\in \bigcup_{i\in\{2,3,4,5,6\}} (C_{i,k}\setminus B_{i,k}) =\bigcup_{i\in\{2,3,4,5,6\}} \{ \min C_{i,k}, \max C_{i,k}\}$.

	Due to~\Cref{lem:tmunique_substring},
	the substring
	$\tm_k[1..\max C_{2,k}-1]$ occurs uniquely and thus
	$\max C_{2,k}\not\in A$ for any
	$A\in \attof{\tm_k}\cap (( \mathcal{C}_k  \otimes (C_{3,k} \cup C_{5,k}))\setminus( \mathcal{B}_k  \otimes (B_{3,k} \cup B_{5,k} )))$
	since it cannot cross any element in $C_{i,k}$ for $i\in\{3,4,5,6\}$.
	Similarly $\min C_{6,k}\not\in A$ for any $A\in \attof{\tm_k}$, due to the substring $\tm_k[\min C_{6,k}+1..|\tm_k|]$,
	since it cannot cross any element in $C_{i,k}$ for $i\in\{2,3,4,5\}$.
	For any $p \in \bigcup_{i\in\{2,3,4,5,6\}} \{\min C_{i,k}, \max C_{i,k} \}$
	(except for $\max C_{2,k}$ and $\min C_{6,k}$ already ruled out above),
	we show that if $p\in A$, there exists a substring that does not cross any position in
	$A$, contradicting that $A$ is a string attractor for $\tm_k$, therefore implying $p\not\in A$.
	We consider two cases depending on whether $k$ is even or odd.

	For even $k\geq 6$, we focus on the following four length-$14$ substrings flanking each $C_{i,k}$:
	\begin{align*}
		\texttt{bbaba}{\overbrace{\underline{\texttt{abba}}}^{C_{4,k}}}\texttt{ababb},
		\texttt{aabab}{\overbrace{\underline{\texttt{baba}}}^{C_{2,k}}}\texttt{ababb},
		\texttt{bbaba}{\overbrace{\underline{\texttt{abab}}}^{C_{6,k}}}\texttt{babaa},
		\texttt{aabab}{\overbrace{\underline{\texttt{baab}}}^{\mathclap{C_{3,k}\mbox{ or }C_{5,k}}}}\texttt{babaa}.
	\end{align*}
	The four (underlined) center symbols of each substring correspond respectively to
	$C_{4,k}$, $C_{2,k}$, $C_{6,k}$, and ($C_{3,k}$ or $C_{5,k}$).
	These are substrings respectively of
	$\tm_{5}\bitcomp{\tm_{5}}$,
	$\tm_{4}\bitcomp{\tm_{4}}$,
	$\bitcomp{\tm_{4}}\tm_{4}$,
	$\bitcomp{\tm_{3}}\tm_{3}$, which are in turn substrings respectively of
	$\tm_{k-1}\bitcomp{\tm_{k-1}}$,
	$\tm_{k-2}\bitcomp{\tm_{k-2}}$,
	$\bitcomp{\tm_{k-2}}\tm_{k-2}$,
	$\bitcomp{\tm_{k-3}}\tm_{k-3}$ for any even $k\geq 6$, where the center positions coincide,
	because $\tm_i$ (resp. $\bitcomp{\tm_i}$) is a prefix and suffix of $\mu^2(\tm_i)=\tm_{i+2}$ (resp. $\mu^2(\bitcomp{\tm_i})=\bitcomp{\tm_{i+2}}$) for any $i\geq 0$.
	Since the lengths of the substrings we consider below are at most $6$,
	any occurrence crossing an element in $\bigcup_{i\in\{2,3,4,5,6\}} C_{i,k}$
	must be a substring of the above strings.
	Therefore, we can see:
	\begin{itemize}
		\item $\min C_{3,k}, \min C_{5,k}\not\in A$ due to the substring $\a\a\b\b\a\b$;
		\item $\max C_{3,k}, \max C_{5,k}\not\in A$ due to the substring $\b\a\b\b\a\a$;
		\item $\min C_{4,k}\not\in A$ due to the substring $\b\b\a\a\b\a$;
		\item $\max C_{4,k}\not\in A$ due to the substring $\a\b\a\a\b\b$;
		\item $\min C_{2,k}\not\in A$ due to the substring $\a\b\a\a$ and the fact that $\min C_{4,k}\not\in A$ and $\min C_{6,k}\not\in A$.
		\item $\max C_{6,k}\not\in A$ due to the substring %
		      $\a\a\b\a$ and the fact that
		      $\max C_{2,k}\not\in A$ and $\max C_{4,k}\not\in A$.
	\end{itemize}
	See also~\Cref{fig:tm-invalid-c-even-k}.
	This proves the lemma for even $k\geq 6$.

	The case for odd $k\geq 7$ is similar (albeit slightly simpler),
	due to the larger separation between the boundaries.
	We focus on the following four length-$22$ substrings flanking each $C_{i,k}$:
	\begin{align*}
		\texttt{aabbaabab}{\overbrace{\underline{\texttt{baba}}}^{C_{4,k}}}\texttt{ababbaabb},
		\texttt{bbaabbaba}{\overbrace{\underline{\texttt{abba}}}^{C_{2,k}}}\texttt{ababbaabb}, \\
		\texttt{aabbaabab}{\overbrace{\underline{\texttt{baab}}}^{C_{6,k}}}\texttt{babaabbaa},
		\texttt{bbaabbaba}{\overbrace{\underline{\texttt{abab}}}^{\mathclap{C_{3,k}\mbox{ or }C_{5,k}}}}\texttt{babaabbaa},
	\end{align*}
	which have occurrences in which the four center symbols of each substring correspond respectively to
	$C_{4,k}$, $C_{2,k}$, $C_{6,k}$, and ($C_{3,k}$ or $C_{5,k}$).
	These are substrings respectively of
	$\tm_{6}\bitcomp{\tm_{6}}$,
	$\tm_{5}\bitcomp{\tm_{5}}$,
	$\bitcomp{\tm_{5}}\tm_{5}$,
	$\bitcomp{\tm_{4}}\tm_{4}$, which are in turn substrings respectively of
	$\tm_{k-1}\bitcomp{\tm_{k-1}}$,
	$\tm_{k-2}\bitcomp{\tm_{k-2}}$,
	$\bitcomp{\tm_{k-2}}\tm_{k-2}$,
	$\bitcomp{\tm_{k-3}}\tm_{k-3}$, for any odd $k\geq 7$, where the center positions coincide.
	Since the lengths of the substrings we consider below are at most $10$,
	any occurrence crossing an element in $\bigcup_{i\in\{2,3,4,5,6\}} C_{i,k}$
	must be a substring of the above strings.
	Therefore, we can see:
	\begin{itemize}
		\item $\min C_{2,k}\not\in A$ due to the substring $\b\b\a\a\b\a$;
		\item $\min C_{3,k}, \min C_{5,k}\not\in A$ due to the substring $\b\a\b\b\a\b\a\a\b\b$;
		\item $\max C_{3,k}, \max C_{5,k}\not\in A$ due to the substring $\a\a\b\b\a\b\a\a\b\a$;
		\item $\min C_{4,k}\not\in A$ due to the substring $\a\b\a\a\b\a\b\b\a\a$;
		\item $\max C_{4,k}\not\in A$ due to the substring $\b\b\a\a\b\a\b\b\a\b$;
		\item $\max C_{6,k}\not\in A$ due to the substring $\b\a\b\b\a\a$.
	\end{itemize}
	See also~\Cref{fig:tm-invalid-c-odd-k}.
	This proves the lemma for odd $k\geq 7$.

\end{proof}

\begin{proof}[Proof of~\Cref{thm:tm_smallest_attractors}]
	The statement can be confirmed by brute force for $n = 6$.
	Let $k \geq 7$ and assume the statement of~\Cref{thm:tm_smallest_attractors}
	holds for all $n$ with $6\leq n < k$.
	\Cref{lem:tm_attractors_completeness}
	implies that if there is an element not in
	$( \mathcal{C}_k  \otimes (C_{3,k}\cup C_{5,k}))$
	but in $\attof{\tm_k}$,
	then there is an element in
	$\attof{\tm_{k-1}}$ not in
	$( \mathcal{B}_{k-1}  \otimes (B_{3,k-1} \cup B_{5,k-1} ))$
	contradicting the induction hypothesis.
	Therefore, $\attof{\tm_k} \subseteq ( \mathcal{C}_k  \otimes (C_{3,k} \cup C_{5,k}))$,
	and thus $\attof{\tm_k} \subseteq ( \mathcal{B}_k  \otimes (B_{3,k} \cup B_{5,k} ))$ by~\Cref{lem:tm_outer_positions}.
	Together with~\Cref{lem:an_attractor_of_tm}, this completes the proof.
\end{proof}

\begin{figure}[t]
	\centering
	\includegraphics[width=\linewidth]{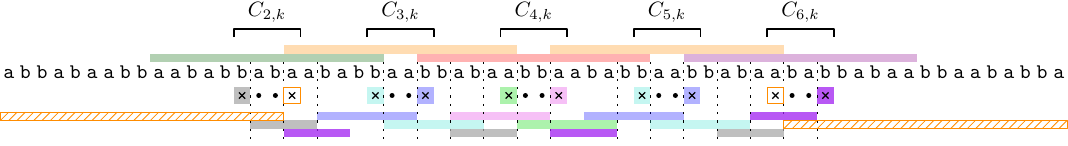}
	\caption{
		Illustration of the proof of \cref{lem:tm_outer_positions} for even $k\geq 6$, shown on $\tm_6$.
		The substrings shown above $\tm_6$ have center four symbols corresponding to
		$C_{4,k}$ (red), $C_{2,k}$ (green), $C_{6,k}$ (purple), and ($C_{3,k}$ or $C_{5,k}$) (yellow).
		For each $i \in \{2,3,4,5,6\}$,
		the elements of $C_{i,k}$ are marked beneath $\tm_6$:
		each ``$\scriptstyle\bullet$'' indicates an element of $B_{i,k}$, and each ``$\scriptstyle\times$'' indicates an element of $C_{i,k} \setminus B_{i,k}$.
		Each colored ``$\scriptstyle\times$'' $\not\in A$ due to the substring shown in the same color, where the orange-shaded substrings are $\tm_k[1..\max C_{2,k}-1]$ and $\tm_k[\min C_{6,k}+1..|\tm_k|]$.
		(We only show occurrences of these substrings that are crossed by some position in $C_{i,k}$.)
	}
	\label{fig:tm-invalid-c-even-k}
\end{figure}

\begin{figure}[t]
	\centering
	\includegraphics[width=\linewidth]{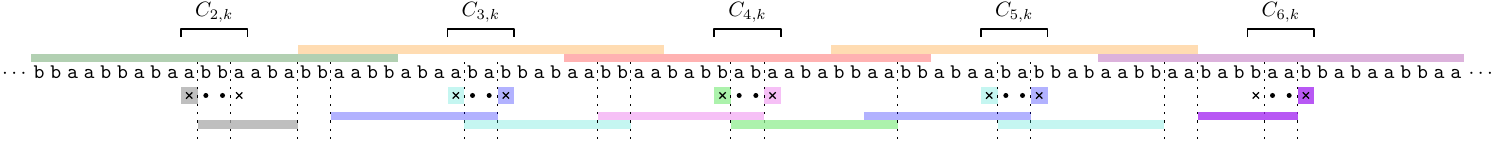}
	\caption{
		Illustration of the proof of \cref{lem:tm_outer_positions} for odd $k$, shown on $\tm_7$ (a prefix and a suffix of $\tm_7$ are omitted, indicated by ``$\cdots$'').
		Legend and color conventions are as in \cref{fig:tm-invalid-c-even-k}.
	}
	\label{fig:tm-invalid-c-odd-k}
\end{figure}

\begin{corollary}
	For $n \geq 6$, $|\attof{\tm_n}| = 32$.
\end{corollary}

\begin{figure}[t]
	\centering
	\includegraphics[width=0.9\linewidth]{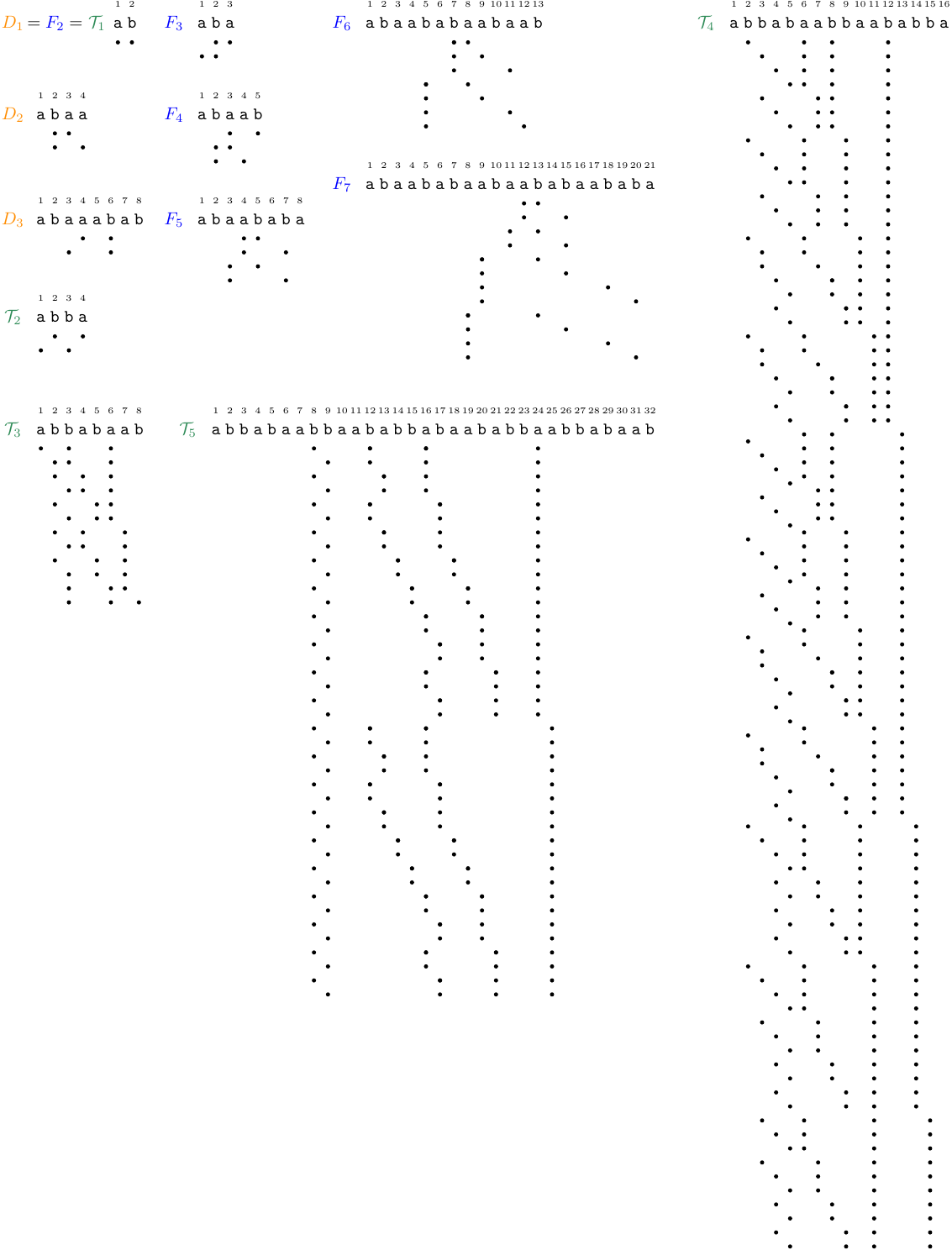}
	\caption{
		Smallest attractors of period-doubling words $D_n$ for $n=1,2,3$, Fibonacci words $F_n$ for $n=2,3,4,5,6,7$, and Thue--Morse words $\tm_n$ for $n=1,2,3,4,5$.
		The visualization pattern follows that of~\Cref{fig:att-example-n=8}.
	}
	\label{fig:attractors-small-orders}
\end{figure}

\section{Discussion}\label{sec:open_problems}
While we studied the smallest attractors of Fibonacci, period-doubling, and Thue--Morse words of higher orders ($F_n$ for $n \geq 7$, $D_n$ for $n \geq 3$, and $\tm_n$ for $n \geq 6$),
for completeness, we list the smallest attractors of lower orders in \Cref{fig:attractors-small-orders}.

Some natural related open problems are:
\begin{itemize}
	\item Characterization of all smallest bidirectional macro schemes (BMS), or all smallest SLPs for the same family of strings considered. (Smallest SLPs for Fibonacci words are known~\cite{DBLP:conf/cpm/MienoIH22}.)
	\item Characterization of all smallest string attractors, BMS, and SLPs for arbitrary prefixes of the considered infinite sequence.
	\item Characterization of all smallest string attractors, BMS, and SLPs of other standard Sturmian words, which include the Fibonacci words.
\end{itemize}

\section*{Acknowledgements}
This work was supported by JSPS KAKENHI Grant Numbers JP24K02899 (HB), JP25K21150 (DK), JP24K20734 (TM), and JP23K10982 (YO).

\section*{{Declaration of generative AI and AI-assisted technologies in the manuscript preparation process}}
During the preparation of this work, the authors used Claude Code (Anthropic; models Claude Opus 5.5 and Claude Fable 5.1)
in order to draft \ref{sec:Walnut-verification} (including \Cref{lem:zeckendorf_l_r} and its proof) and parts of \Cref{sec:open_problems},
to revise the related text in the introduction and in \ref{sec:Walnut}, including figure captions,
and to proofread the manuscript, proposing corrections and clarifications of definitions, proofs, and figure captions.
After using this tool, the authors reviewed and edited the content as needed and take full responsibility for the content of the publication.

\clearpage

\clearpage

\appendix
\section{Using Walnut to Characterize the Smallest Attractors}\label{sec:Walnut}
Using the approach by Schaeffer and Shallit~\cite{schaeffer2024stringattractorsautomaticsequences},
we can characterize the smallest attractors of period-doubling and Fibonacci words using the following Walnut program.
Below, we first introduce the connection,
and then, in~\ref{sec:Walnut-verification}, use the resulting automata to
mechanically verify~\Cref{thm:pd_attractors,thm:all_smallest_attractors_of_fib,thm:tm_smallest_attractors} for all orders in their stated ranges.

\begin{verbatim}
def pdfaceq "k+i>=j & A u,v (u>=i & u<=j & v+j=u+k) => PD[u]=PD[v]":
[strategy * CCLS] def pdsa2 "(i1<n) & (i2<n) & Ak,l (k<=l & l<n) => (Er,s r<=s & s<n &
 (s+k=r+l) & $pdfaceq(k,l,s) & ((r<=i1 & i1<=s) | (r<=i2 & i2<=s)))"::
def pdfa1 "An (n>=2) => Ei1,i2 $pdsa2(i1,i2,n)":
def pdfa2 "$pdsa2(i1,i2,n) & i1 < i2":

def fibfaceq "?msd_fib k+i>=j & A u,v (u>=i & u<=j & v+j=u+k) => F[u]=F[v]":
[strategy * CCLS] def fibsa2 "?msd_fib (i1<n) & (i2<n) & Ak,l (k<=l & l<n) =>
 (Er,s r<=s & s<n & (s+k=r+l) & $fibfaceq(k,l,s)
 & ((r<=i1 & i1<=s) | (r<=i2 & i2<=s)))"::
def fibfa1 "?msd_fib An (n>=2) => Ei1,i2 $fibsa2(i1,i2,n)":
def fibfa2 "?msd_fib $fibsa2(i1,i2,n) & i1 < i2":
\end{verbatim}

The code for the automaton \texttt{pdfa2} for period-doubling words is taken from~\cite{schaeffer2024stringattractorsautomaticsequences}, except that we added the metacommand \texttt{[strategy * CCLS]}, which selects the determinization strategy CCLS~\cite{DBLP:conf/tacas/NicolF26} (see below).
The code for \texttt{fibfa2} only changes \texttt{PD} (period-doubling words)
to \texttt{F} (Fibonacci words), and specifies the numeration system
the corresponding automaton is defined for.
Walnut returns \texttt{TRUE} for \texttt{pdfa1} and \texttt{fibfa1}, i.e., every prefix of length at least $2$ of $D_\infty$ and of $F_\infty$ has an attractor of size $2$ (for $D_\infty$, this was shown in~\cite{schaeffer2024stringattractorsautomaticsequences}); hence, \texttt{pdfa2} and \texttt{fibfa2} accept exactly the smallest attractors of these prefixes.

See~\Cref{fig:pdfa2,fig:fibfa2}. While these automata characterize the set of all
smallest string attractors of arbitrary prefixes of the period-doubling words
or Fibonacci words, their interpretation is not straightforward.
To limit the lengths of prefixes to the words we consider, we can modify the above to:
\begin{verbatim}
reg pow2 msd_2 "10*":
[strategy * CCLS] def pdfa2pow2 "(i1<n) & (i2<n) & (i1<i2) & $pow2(n) &
 Ak,l (k<=l & l<n) => (Er,s r<=s & s<n & (s+k=r+l) & $pdfaceq(k,l,s)
 & ((r<=i1 & i1<=s) | (r<=i2 & i2<=s)))"::

reg powfib msd_fib "10*":
[strategy * CCLS] def fibfa2powfib "?msd_fib (i1<n) & (i2<n) & (i1<i2) & $powfib(n) &
 Ak,l (k<=l & l<n) => (Er,s r<=s & s<n & (s+k=r+l) & $fibfaceq(k,l,s)
 & ((r<=i1 & i1<=s) | (r<=i2 & i2<=s)))"::
\end{verbatim}
\Cref{fig:pdfa-pow2} shows the sub-automaton \texttt{pdfa2pow2} of \texttt{pdfa2} for paths corresponding to prefixes of $D_\infty$ of lengths $2^k$, and can be used to verify the results of~\Cref{thm:pd_attractors}.
For Fibonacci words, \Cref{fig:fibfa2-filtered} shows the automaton \texttt{fibfa2powfib}, which is equivalent to the sub-automaton of \texttt{fibfa2}
containing only the accepting paths for prefixes of length $f_k$,
which can be used to verify the results of~\Cref{thm:all_smallest_attractors_of_fib}.

A similar program can be made for Thue--Morse words and $4$ attractor positions:
\begin{verbatim}
def tmfaceq "k+i>=j & A u,v (u>=i & u<=j & v+j=u+k) => T[u]=T[v]":
[strategy * CCLS] def tmfa4 "(i1<n) & (i2<n) & (i3<n) & (i4<n) & (i1<i2)
 & (i2<i3) & (i3<i4) & Ak,l (k<=l & l<n) => (Er,s r<=s & s<n
  & (s+k=r+l) & $tmfaceq(k,l,s) & ((r<=i1 & i1<=s) | (r<=i2 & i2<=s)
  | (r <=i3 & i3<=s) | (r<=i4 & i4<=s)))"::
\end{verbatim}
Here, the \texttt{[strategy]} metacommand is essential, as the computation was not able to finish under the default strategy.
The computation took around 6.5 hours on an Intel Xeon Platinum 8462Y+ processor (2.8GHz), using 705GiB of memory.
The resulting automaton \texttt{tmfa4}, analogous to \texttt{pdfa2} and \texttt{fibfa2}, has 720 states (\Cref{fig:tmfa4}).
We note that the CCLS strategy also speeds up the computations for the period-doubling and Fibonacci words considerably: with the metacommand \texttt{[strategy * CCLS]}, the computations of \texttt{pdsa2} and \texttt{fibsa2} took less than one second and about four seconds, respectively, whereas they took about 20 minutes and 2.5 hours with the default strategy.
To limit the lengths of prefixes to $2^k$, we add the restriction \texttt{\$pow2(n)} as above:
\begin{verbatim}
[strategy * CCLS] def tmfa4pow2 "(i1<n) & (i2<n) & (i3<n) & (i4<n) & (i1<i2)
 & (i2<i3) & (i3<i4) & $pow2(n) & Ak,l (k<=l & l<n) => (Er,s r<=s & s<n
  & (s+k=r+l) & $tmfaceq(k,l,s) & ((r<=i1 & i1<=s) | (r<=i2 & i2<=s)
  | (r <=i3 & i3<=s) | (r<=i4 & i4<=s)))"::
\end{verbatim}
The computation took roughly the same amount of time and memory as \texttt{tmfa4}.
See~\Cref{fig:tmsa-pow2}.

\subsection{Mechanical Verification of the Main Theorems}\label{sec:Walnut-verification}
Since the automata \texttt{pdfa2pow2}, \texttt{fibfa2powfib}, and \texttt{tmfa4pow2}
(\Cref{fig:pdfa-pow2,fig:fibfa2-filtered,fig:tmsa-pow2}) accept exactly
the (representations of the) attractors of $D_k$, $F_k$, and $\tm_k$ of the respective sizes,
our theorems can be verified for all orders in their stated ranges
($k\geq 3$ for $D_k$, $k\geq 7$ for $F_k$, and $k\geq 6$ for $\tm_k$) by showing that
the sets described by our theorems are regular,
and that the corresponding automata accept the same languages.

The automata read words over the alphabet $\{0,1\}^d$, where $d=3$ for $D_k$ and $F_k$, and $d=5$ for $\tm_k$.
For $1\leq j\leq d$, the binary word formed by the $j$th components of the symbols of such a word is called its $j$th \emph{track}.
The first $d-1$ tracks represent the attractor positions in increasing order, and the last track, the \emph{length track}, represents the length of the prefix (cf.~\Cref{fig:pdfa2,fig:fibfa2}).

For an integer $x\geq 0$ and $m\geq 1$,
let $\llbracket x\rrbracket_m$ denote the length-$m$ binary representation of $x$
(most significant digit first, with leading zeros), which is used for $D_k$ and $\tm_k$,
and let $\langle x\rangle_m$ denote the length-$m$ Zeckendorf representation of $x$, which is used for $F_k$, i.e.,
$\langle x\rangle_m = d_{m-1}\cdots d_0$ with $x = \sum_{j} d_j f_{j+1}$ and no two consecutive $1$s.
Recall that Walnut indexes positions from $0$, so a position $p$ is represented by $\llbracket p-1\rrbracket_m$ or $\langle p-1\rangle_m$,
and an accepted word may be preceded by an arbitrary number of copies of the symbol $(0,\ldots,0)\in\{0,1\}^d$.
The length track is $\llbracket 2^k\rrbracket_{k+1} = 10^{k}$ for $D_k$ and $\tm_k$,
and $\langle f_k\rangle_{k} = 10^{k-1}$ for $F_k$.

\paragraph{Period-doubling words}
For $k\geq 3$,
$\llbracket p_k-1\rrbracket_{k+1} = 00101^{k-3}$,
$\llbracket q_k-1\rrbracket_{k+1} = 00111^{k-3}$, and
$\llbracket r_k-1\rrbracket_{k+1} = 01011^{k-3}$.

\paragraph{Thue--Morse words}
For $k\geq 6$ and $2\leq i\leq 6$, the two positions $i\cdot 2^{k-3}$ and $i\cdot 2^{k-3}+1$ of $B_{i,k}$
are represented by
$\llbracket i\cdot 2^{k-3}-1\rrbracket_{k+1} = \llbracket i-1\rrbracket_4\,1^{k-3}$ and
$\llbracket i\cdot 2^{k-3}\rrbracket_{k+1} = \llbracket i\rrbracket_4\,0^{k-3}$.

\paragraph{Fibonacci words}
For $m\geq 1$, let $a_m = \langle f_m-1\rangle_m$, which is the length-$m$ prefix of $(01)^\infty$.
Let $X_2 = \{00\}$, $X_3=\{001\}$, and
$X_k = \{0a_{k-1}\} \cup 0\{0,1\}X_{k-2}$ for $k\geq 4$.
Notice that every word in $X_k$ starts with $0$.
For a set $P$ of integers, let $\langle P\rangle_m = \{\langle x\rangle_m : x\in P\}$
and $P\ominus 1 = \{x-1 : x\in P\}$.
\begin{lemma}\label{lem:zeckendorf_l_r}
	For $k\geq 2$, $\langle (L_k\cup R_k)\ominus 1\rangle_{k+1} = \{a_{k+1}\}\cup 1X_k$.
	Furthermore, for $k\geq 4$,
	\begin{align*}
		\langle L_k\ominus 1\rangle_{k+1} & = \{a_{k+1}\}\cup 100X_{k-2},   &
		\langle R_k\ominus 1\rangle_{k+1} & = \{10a_{k-1}\}\cup 101X_{k-2},
	\end{align*}
	and $\langle L'_k\ominus 1\rangle_{k+1} = \{a_{k+1}, 10^k\}$ if $k$ is even and $\{a_{k+1}\}$ if $k$ is odd.
\end{lemma}
\begin{proof}
	We proceed by induction on $k$: assuming the first statement for $k-2$, we derive the representations of $L_k\ominus 1$ and $R_k\ominus 1$, whose union gives the first statement for $k$.
	The first statement holds for $k=2,3$ since
	$(L_2\cup R_2)\ominus 1 = \{2,3\}$ with $\langle\{2,3\}\rangle_3 = \{010, 100\}$,
	and $(L_3\cup R_3)\ominus 1 = \{4,6\}$ with $\langle\{4,6\}\rangle_4 = \{0101, 1001\}$.
	Let $k\geq 4$ and assume the first statement holds for $k-2$.
	Let $u\in (L_{k-2}\cup R_{k-2})\ominus 1$;
	by~\Cref{obs:l_r_min_max}, $f_{k-1}-1\leq u\leq f_k-2$, so $\langle u \rangle_{k-1}$ is well defined.
	By~\Cref{eqn:l_n}, $L_k\ominus 1 = f_{k}\oplus((L_{k-2}\cup R_{k-2})\ominus 1)$.
	If $\langle u\rangle_{k-1} = a_{k-1}$, i.e., $u = f_{k-1}-1$, then $f_k+u = f_{k+1}-1$, represented by $a_{k+1}$.
	Otherwise, $\langle u\rangle_{k-1} = 1x$ for some $x\in X_{k-2}$,
	i.e., $u = f_{k-1}+u'$ with $\langle u'\rangle_{k-2} = x$,
	and $u' < f_{k-2}$ since $x$ starts with $0$.
	Then, $f_k+u = f_{k+1}+u'$ is represented by $10\langle u'\rangle_{k-1} = 100x$.
	By the induction hypothesis, this gives $\langle L_k\ominus 1\rangle_{k+1} = \{a_{k+1}\}\cup 100X_{k-2}$.
	By~\Cref{eqn:r_n}, $R_k\ominus 1 = f_{k+1}\oplus((L_{k-2}\cup R_{k-2})\ominus 1)$,
	and since $u < f_k$, $\langle f_{k+1}+u\rangle_{k+1} = 10\langle u\rangle_{k-1}$.
	This gives, again by the induction hypothesis, $\langle R_k\ominus 1\rangle_{k+1} = 10(\{a_{k-1}\}\cup 1X_{k-2}) = \{10a_{k-1}\}\cup 101X_{k-2}$.
	The union of the two representations is
	$\{a_{k+1}\}\cup 1(\{0a_{k-1}\}\cup 00X_{k-2}\cup 01X_{k-2}) = \{a_{k+1}\}\cup 1X_k$.
	Finally, by~\Cref{eqn:l_n_prime}, $L'_k\ominus 1 = \{f_{k+1}-1, f_{k+1}\}$ if $k$ is even and $\{f_{k+1}-1\}$ otherwise,
	represented by $a_{k+1}$ and $10^k$.
\end{proof}
By~\Cref{thm:all_smallest_attractors_of_fib}, for $k\geq 7$,
the attractors $\{p,q\}\in\attof{F_k}$ with $p<q$ correspond to the words whose tracks
(starting at the position of the leading $1$ of $\langle f_k\rangle_k$) are
$00\langle p-1\rangle_{k-2}$, $0\langle q-1\rangle_{k-1}$, and $10^{k-1}$,
where by~\Cref{lem:zeckendorf_l_r},
either $\langle p-1\rangle_{k-2}\in\{a_{k-2}\}\cup 1X_{k-3}$ and $\langle q-1\rangle_{k-1} \in \{a_{k-1}\}\cup 100X_{k-4}$,
or $\langle p-1\rangle_{k-2}\in \{a_{k-2}\}\cup\{10^{k-3} : k \text{ odd}\}$ and $\langle q-1\rangle_{k-1}\in\{10a_{k-3}\}\cup 101X_{k-4}$.

\paragraph{Verification}
All the sets above are regular:
$\bigcup_m\{a_m\} = (01)^*(0+\varepsilon)$ and
$\bigcup_{k\geq 2}X_k = (0(0+1))^*(0\,010(10)^*(1+\varepsilon)+00+001)$.
Moreover, since all tracks are read synchronously from the position of the leading~$1$ of the length track,
whose position determines the order~$k$,
each family of sets in~\Cref{thm:pd_attractors,thm:all_smallest_attractors_of_fib,thm:tm_smallest_attractors}
(for $k\geq 3$, $k\geq 7$, and $k\geq 6$, respectively)
is recognized by a deterministic finite automaton obtained as a product of automata for the individual tracks.
We constructed these automata and checked, by a breadth-first search on the product automaton,
that each accepts exactly the same language as
\texttt{pdfa2pow2}, \texttt{fibfa2powfib}, and \texttt{tmfa4pow2}, respectively,
restricted to the same ranges of orders
(the restrictions of \texttt{pdfa2}, \texttt{fibfa2}, and \texttt{tmfa4} to lengths $2^k$ and $f_k$ also agree).

Together with the remaining small orders, which can be checked by brute force
(see also~\Cref{fig:attractors-small-orders}),
this gives an independent, mechanical confirmation
of~\Cref{thm:pd_attractors,thm:all_smallest_attractors_of_fib,thm:tm_smallest_attractors},
relying only on the known sizes of the smallest attractors of these words, and on the correctness of Walnut and of the (simple) automata constructions above.
The programs used for this verification (a self-contained one of a few hundred lines of Python, requiring only the standard library,
and a shorter one of about a hundred lines using the library automata-lib)
were written with the assistance of Claude Code (Anthropic)
and are publicly available together with the Walnut commands and outputs at \url{https://github.com/hdbn/allstringattractors}.

\begin{figure}
	\centering
	\includegraphics[height=0.7\textheight]{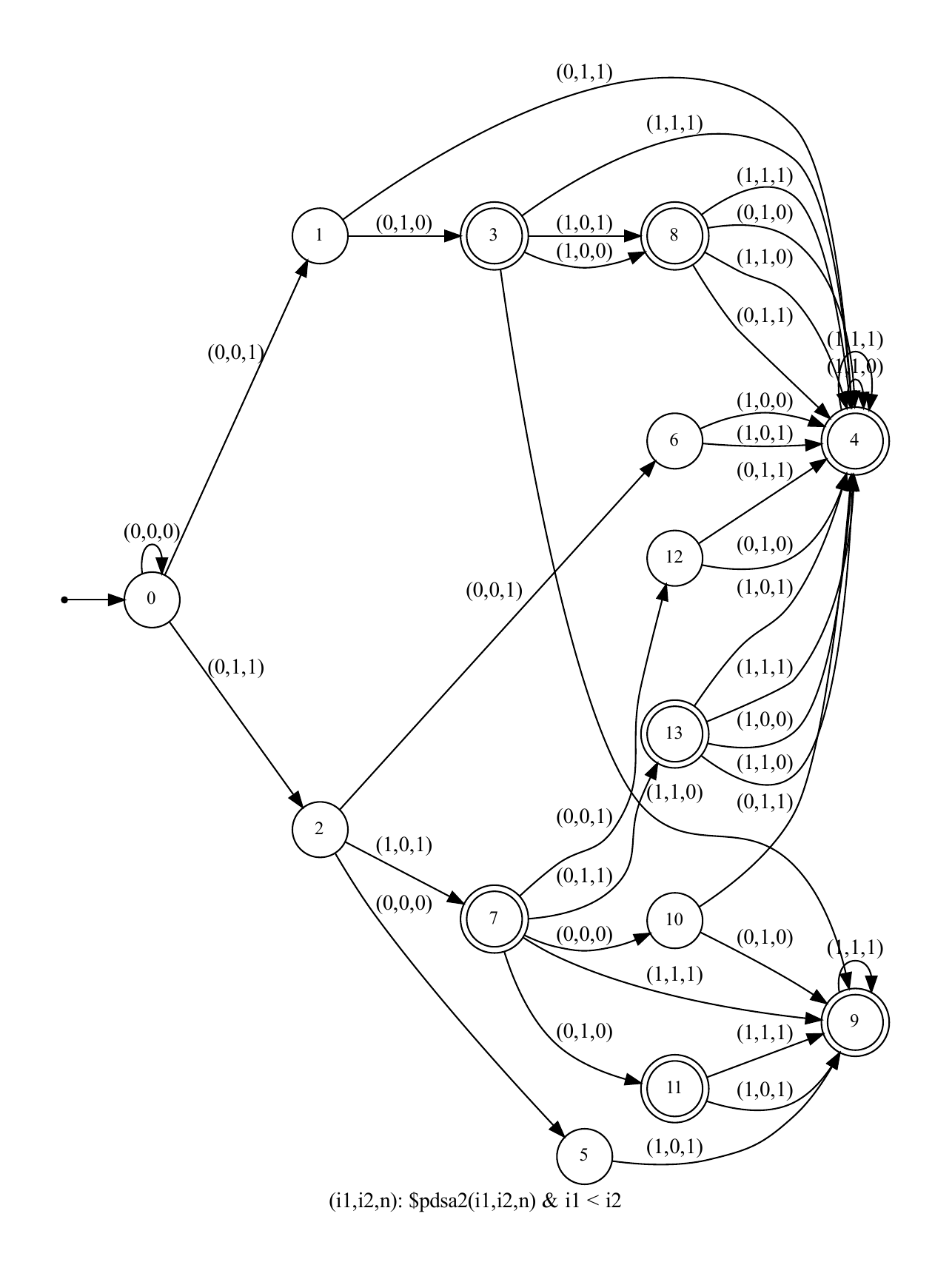}
	\caption{The automaton \texttt{pdfa2} accepting a sequence $s\in (\{0,1\}\times\{0,1\}\times\{0,1\})^*$ if and only if $s[i] = (\mathsf{p}[i],\mathsf{q}[i],\mathsf{n}[i])$ where $\mathsf{p}, \mathsf{q}, \mathsf{n}$ are respectively the binary
		representation of integers $p-1$, $q-1$, $n$,
		satisfying $\{p, q\} \in \attof{D_\infty[1..n]}$ and $p < q$.
		Here, the $-1$ is due to the difference in the start index used for strings.
		By analyzing
		accepting paths for which $\mathsf{n} = 10^k$ implying $n = 2^k$
		we can verify the statement of~\Cref{thm:pd_attractors}.
		See also~\Cref{fig:pdfa-pow2}.}\label{fig:pdfa2}
\end{figure}

\begin{figure}
	\centering
	\includegraphics[width=0.8\textwidth]{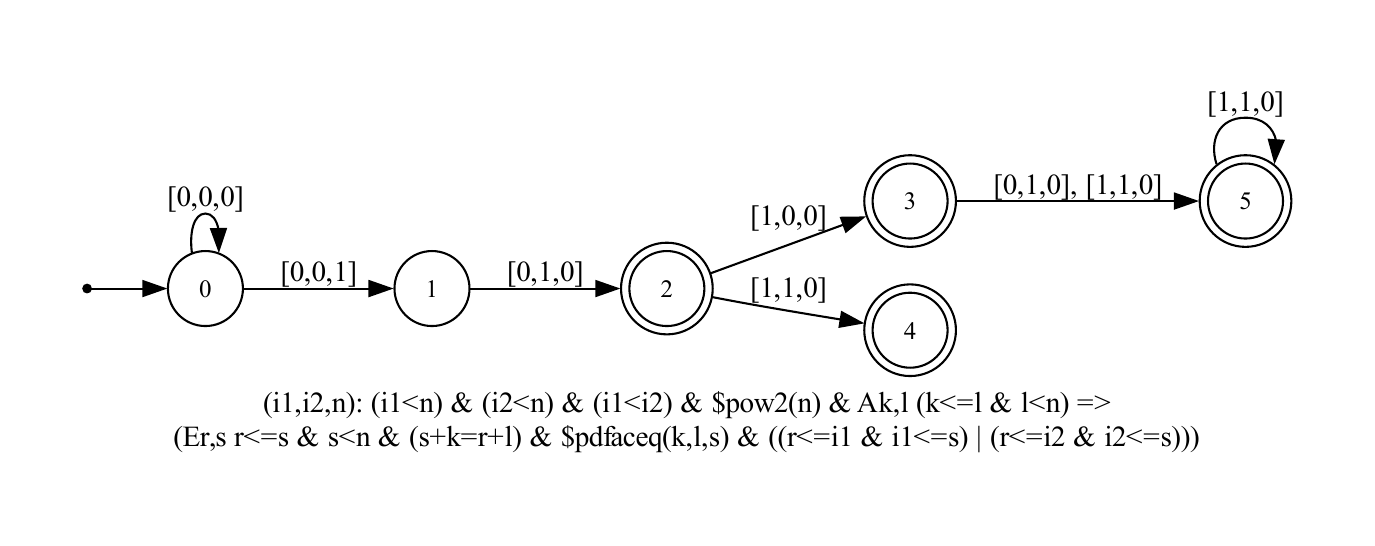}
	\caption{The automaton \texttt{pdfa2pow2} obtained by removing, from the automaton in~\Cref{fig:pdfa2}, all edges and nodes that are not part of
		accepting paths with $\mathsf{n}=10^k$ implying $n = 2^k$.
		We can see that
		for $k = 2$,
		$\mathsf{p} = 001$ implying $p=2$ and
		$\mathsf{q} = 010 \mbox{ or } 011$ implying $q=3 \mbox{ or } 4$,
		i.e., $\attof{D_k} = \{\{ 2,3 \}, \{ 2,4 \} \}$.
		For $k \geq 3$,
		$\mathsf{p} = 00101^{k-3} \mbox{ or } 00111^{k-3}$ implying
		$p = 3\cdot 2^{k-3} = p_k\mbox{ or } 2^{k-1} = q_k$,
		and $\mathsf{q} = 01011^{k-3}$ implying
		$q = 3\cdot2^{k-2} = r_k$
		(where $p_k$, $q_k$, $r_k$ are as in~\Cref{thm:pd_attractors}), i.e.,
		$\attof{D_k} = \{ \{  3\cdot 2^{k-3}, 3\cdot2^{k-2}\}, \{ 2^{k-1}, 3\cdot2^{k-2}\} \}$,
		verifying the statement of~\Cref{thm:pd_attractors}.
	}\label{fig:pdfa-pow2}
\end{figure}

\begin{figure}
	\centering
	\includegraphics[height=0.7\textheight]{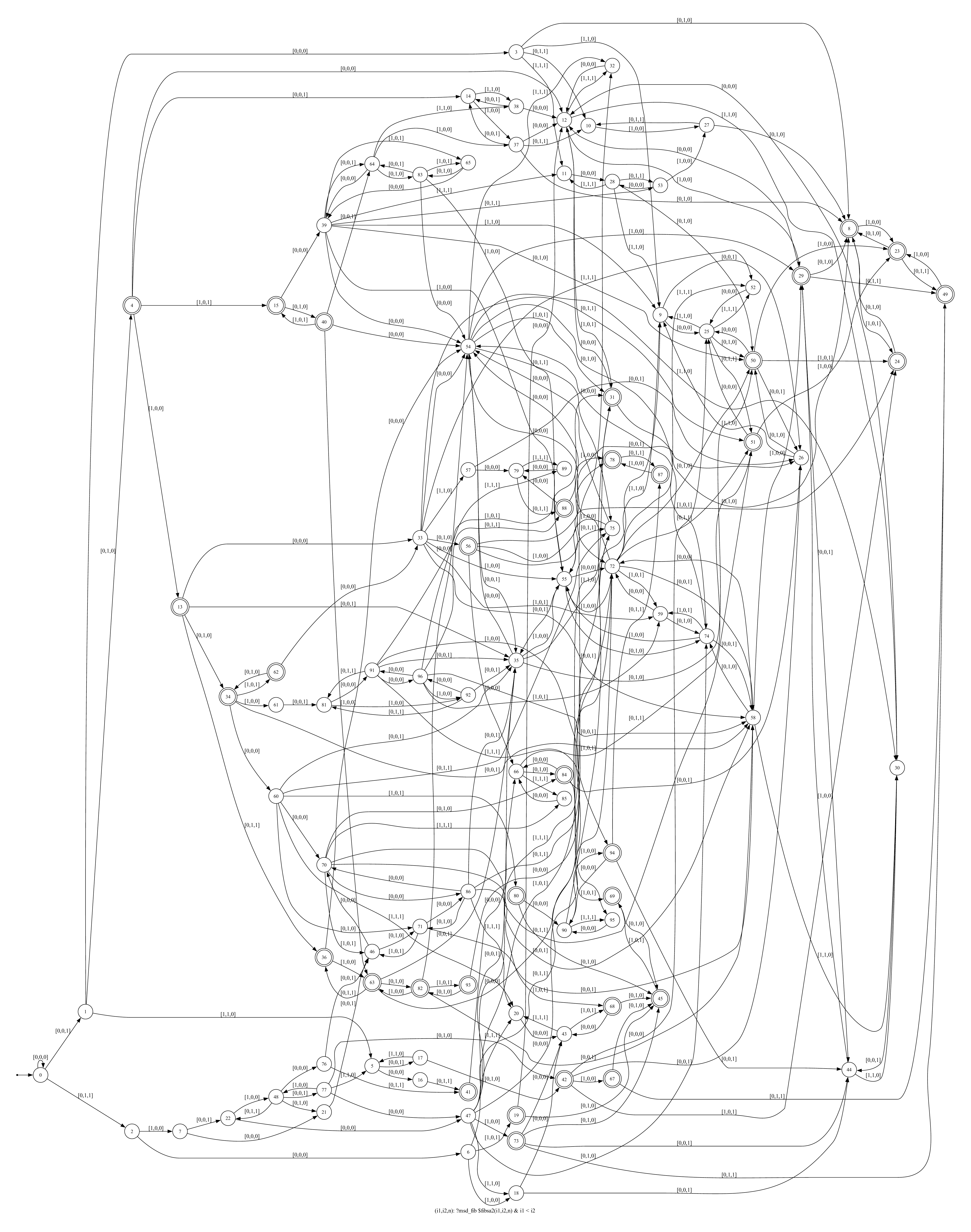}
	\caption{The automaton \texttt{fibfa2} accepting a sequence $s\in (\{0,1\}\times\{0,1\}\times\{0,1\})^*$ if and only if $s[i] = (\mathsf{p}[i],\mathsf{q}[i],\mathsf{n}[i])$ where $\mathsf{p}, \mathsf{q}, \mathsf{n}$ are respectively the
		bit representation, in Fibonacci numeration (Zeckendorf bit representation),
		of integers $p-1$, $q-1$, $n$ satisfying
		$\{p,q\} \in \attof{F_\infty[1..n]}$ and $p < q$.
		Here, the $-1$ is due to the difference in the start index used for strings.
		The accepting paths for which $\mathsf{n} = 10^k$ implies
		$n = f_{k+1}$ and the values $p$ and $q$ implied from $\mathsf{p}$ and $\mathsf{q}$ should correspond to the statement of~\Cref{thm:all_smallest_attractors_of_fib}.
		See~\Cref{fig:fibfa2-filtered} for a simplified automaton containing only such paths.
	}\label{fig:fibfa2}
\end{figure}

\begin{figure}
	\centering
	\includegraphics[width=\textwidth]{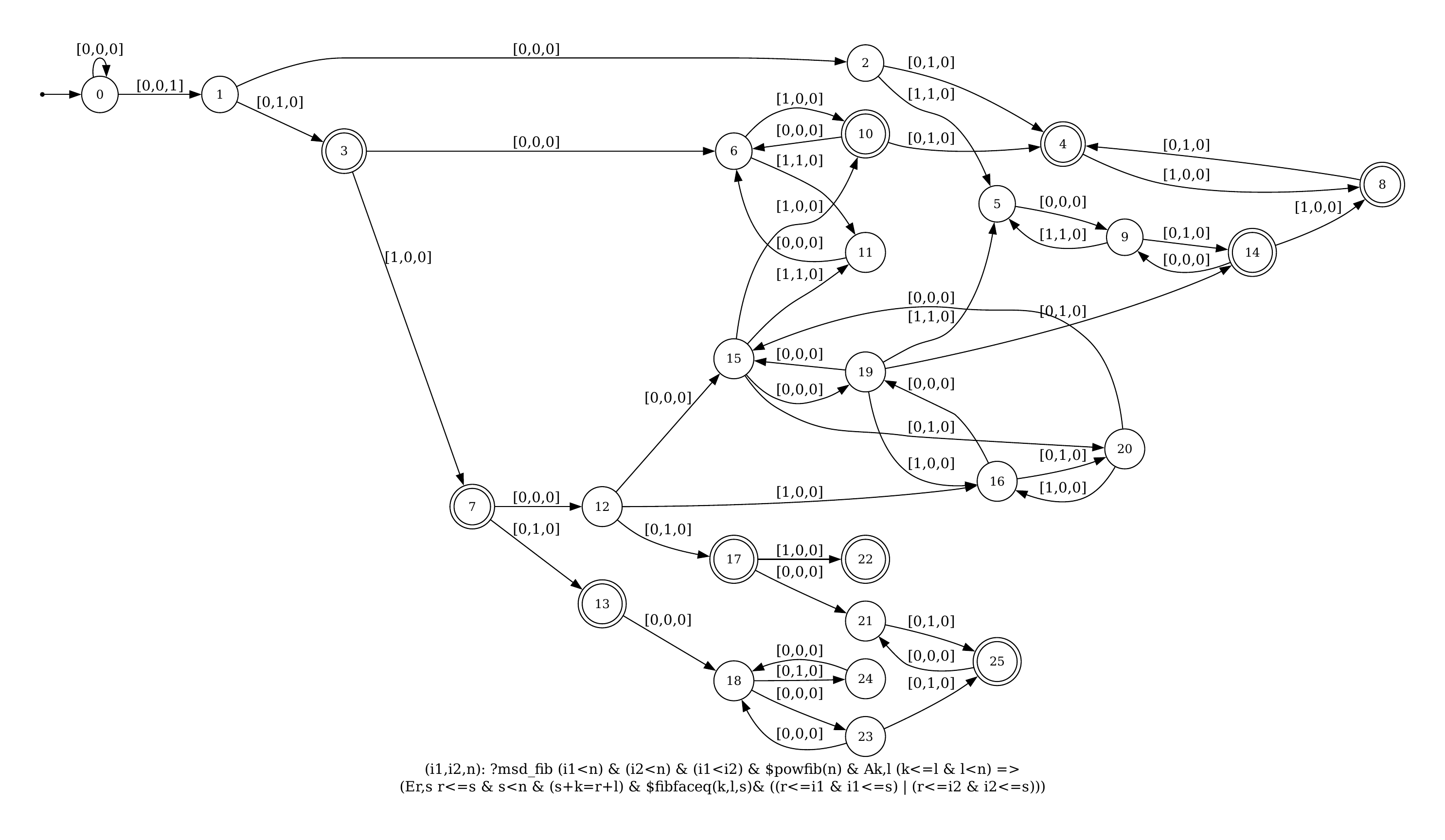}
	\caption{The automaton \texttt{fibfa2powfib}, accepting the same language as the automaton obtained by removing, from the automaton in~\Cref{fig:fibfa2}, all edges and nodes that are not part of
		accepting paths with $\mathsf{n}=10^k$ representing $f_{k+1}$ in the Fibonacci numeration.
		Each accepting path of length $k+1$ starting with $[0,0,1]$
		represents a position pair in $\attof{F_{k+1}}$, and the automaton can be used to verify~\Cref{thm:all_smallest_attractors_of_fib}
		(see~\ref{sec:Walnut-verification}).
	}\label{fig:fibfa2-filtered}
\end{figure}

\begin{figure}[p]
	\centering\includegraphics[height=0.9\textheight]{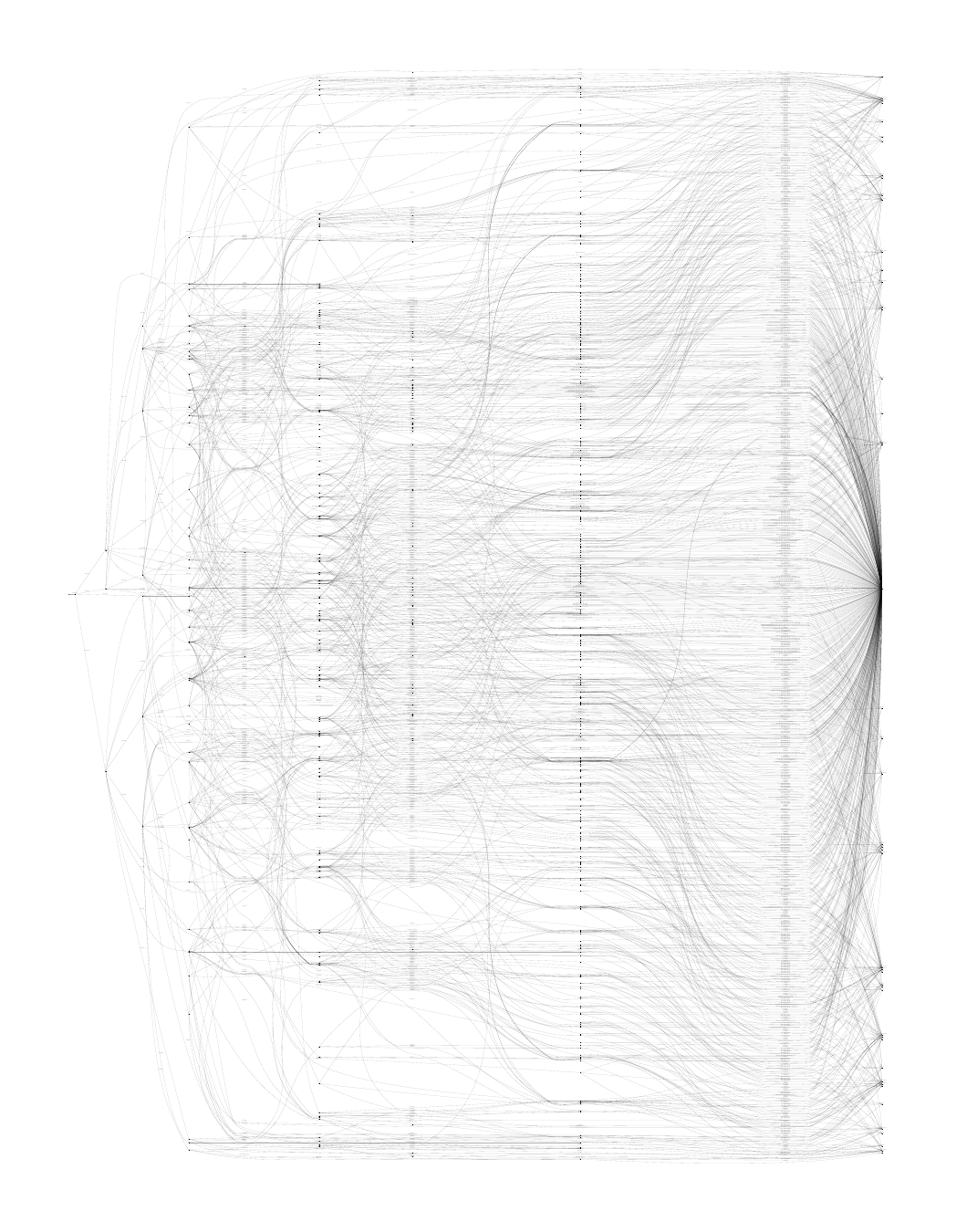}
	\caption{The automaton \texttt{tmfa4} ($720$ states) accepting a sequence $s\in (\{0,1\}^5)^*$ if and only if
		$s[i] = (\mathsf{p}_1[i],\mathsf{p}_2[i],\mathsf{p}_3[i],\mathsf{p}_4[i],\mathsf{n}[i])$
		where $\mathsf{p}_1,\mathsf{p}_2,\mathsf{p}_3,\mathsf{p}_4,\mathsf{n}$ are respectively the binary
		representation of integers $p_1-1$, $p_2-1$, $p_3-1$, $p_4-1$, $n$,
		such that $\{p_1,p_2,p_3,p_4\}$ is an attractor of $\tm_\infty[1..n]$ and $p_1<p_2<p_3<p_4$.
		Here, the $-1$ is due to the difference in the start index used for strings.
		Since the automaton has $720$ states, it cannot be read from the figure, in which states are drawn as dots and edge labels are omitted;
		the automaton is available in machine-readable form in the repository (see~\ref{sec:Walnut-verification}).
		See~\Cref{fig:tmsa-pow2} for the automaton \texttt{tmfa4pow2} containing only the accepting paths for prefixes of length $2^k$.}\label{fig:tmfa4}
\end{figure}

\begin{figure}[h]
	\centering\includegraphics[height=0.9\textheight]{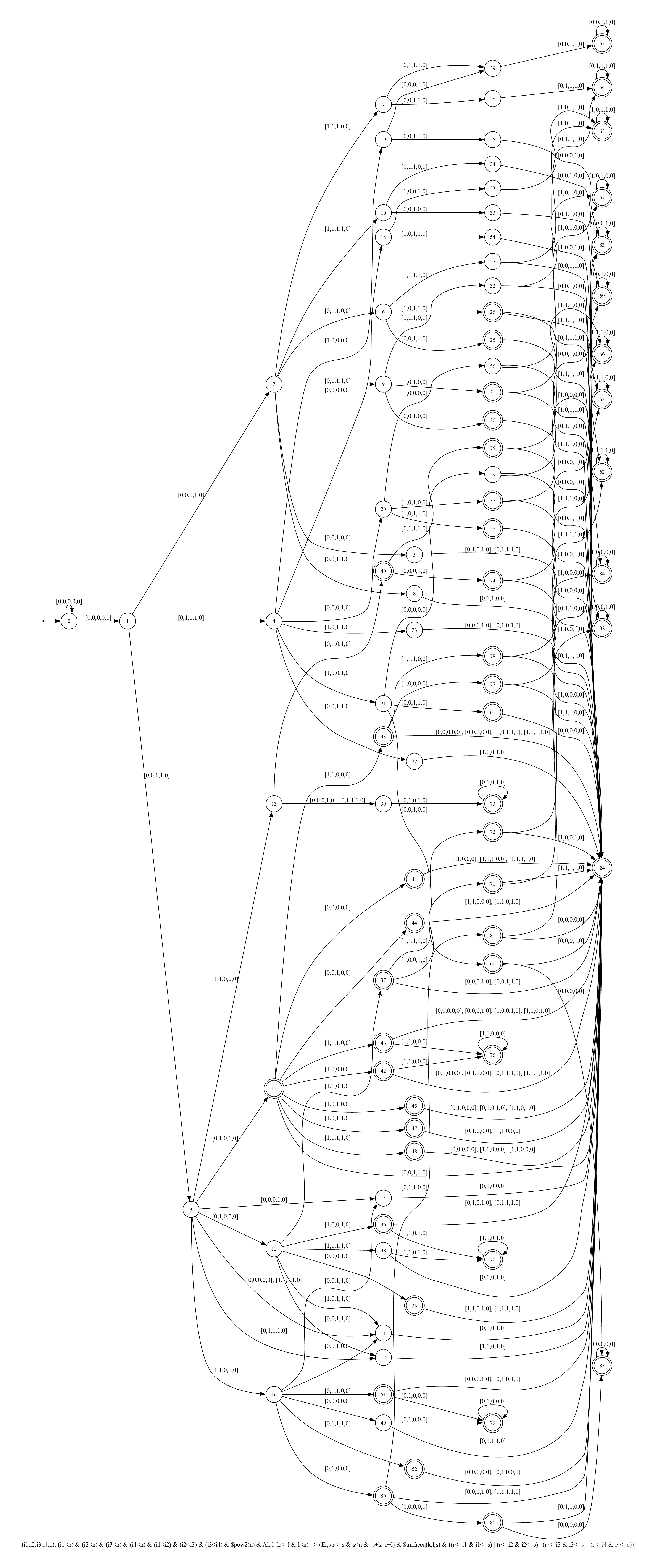}
	\caption{
		The automaton \texttt{tmfa4pow2} whose accepting paths correspond to $4$ distinct positions that are attractors of the length $2^k$ prefix of the Thue--Morse word.
		It can be verified that there are 32 accepting paths that start with $[0,0,0,0,1]$
		for each length $k+1 \geq 7$ (corresponding to $\mathsf{n} = 10^{k}$, i.e. $2^k$),
		and the attractor positions correspond to those in the statement of~\Cref{thm:tm_smallest_attractors}
		(see~\ref{sec:Walnut-verification}).
		Since $\gamma(\tm_k)=4$ for $k\geq 4$, the accepted sets are exactly the smallest attractors for $k\geq 4$;
		for $k\leq 3$, they are attractors that are not smallest.
	}\label{fig:tmsa-pow2}
\end{figure}

\clearpage
\section{Additional Figures}\label{sec:additional-figures}

\begin{figure}[h]
	\centering
	\includegraphics[width=0.78\linewidth]{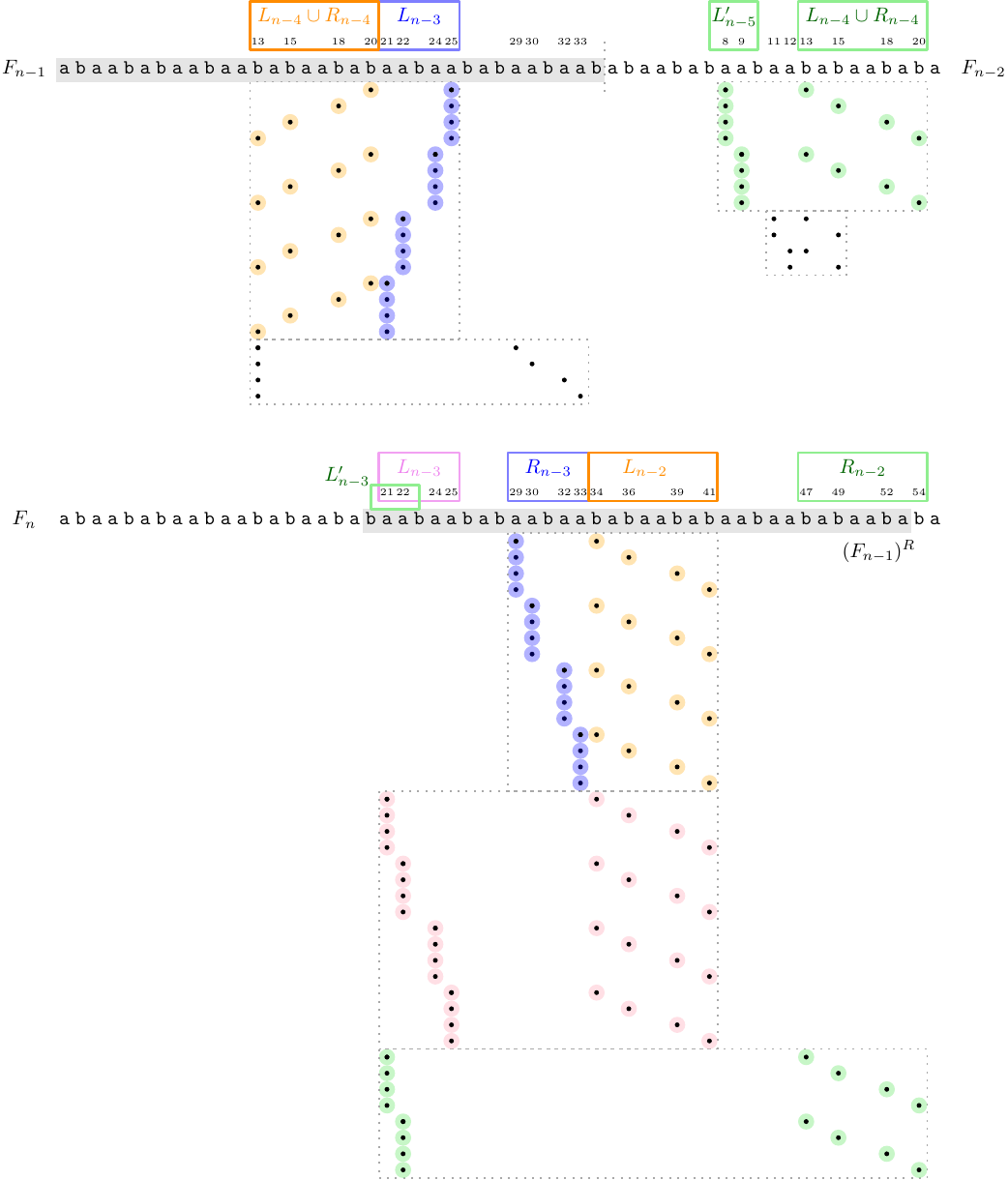}
	\caption{
		Extended version of \Cref{fig:att-example-n=8} for the case $n=9$.
		For each of $F_n$ (bottom), $F_{n-1}$ (top left), and $F_{n-2}$ (top right),
		all of their smallest attractors are indicated as pairs of horizontally aligned dots.
		Characterization of the attractors of $F_n$ via \Cref{thm:all_smallest_attractors_of_fib}
		is illustrated by colors:
		the blue and yellow dots illustrate the mapping via the occurrences of $F_{n-1}$ and $(F_{n-1})^R$ in \Cref{lem:valid_att_from_flip};
		the pink dots illustrate the offset in \Cref{lem:valid_att_from_flip_offset};
		while the green dots illustrate \Cref{lem:valid_att_from_n-2_offset}.
	}
	\label{fig:att-example-n=9}
\end{figure}


\begin{thebibliography}{10}
\expandafter\ifx\csname url\endcsname\relax
  \def\url#1{\texttt{#1}}\fi
\expandafter\ifx\csname urlprefix\endcsname\relax\def\urlprefix{URL }\fi
\expandafter\ifx\csname href\endcsname\relax
  \def\href#1#2{#2} \def\path#1{#1}\fi

\bibitem{DBLP:conf/stoc/KempaP18}
D.~Kempa, N.~Prezza, At the roots of dictionary compression: string attractors,
  in: I.~Diakonikolas, D.~Kempe, M.~Henzinger (Eds.), Proceedings of the 50th
  Annual {ACM} {SIGACT} Symposium on Theory of Computing, {STOC} 2018, Los
  Angeles, CA, USA, June 25-29, 2018, {ACM}, 2018, pp. 827--840.
\newblock \href {https://doi.org/10.1145/3188745.3188814}
  {\path{doi:10.1145/3188745.3188814}}.

\bibitem{DBLP:conf/ictcs/MantaciRRRS19}
S.~Mantaci, A.~Restivo, G.~Romana, G.~Rosone, M.~Sciortino,
  \href{https://ceur-ws.org/Vol-2504/paper8.pdf}{String attractors and
  combinatorics on words}, in: A.~Cherubini, N.~Sabadini, S.~Tini (Eds.),
  Proceedings of the 20th Italian Conference on Theoretical Computer Science,
  {ICTCS} 2019, Como, Italy, September 9-11, 2019, Vol. 2504 of {CEUR} Workshop
  Proceedings, CEUR-WS.org, 2019, pp. 57--71.
\newline\urlprefix\url{https://ceur-ws.org/Vol-2504/paper8.pdf}

\bibitem{DBLP:journals/tcs/MantaciRRRS21}
S.~Mantaci, A.~Restivo, G.~Romana, G.~Rosone, M.~Sciortino, A combinatorial
  view on string attractors, Theor. Comput. Sci. 850 (2021) 236--248.
\newblock \href {https://doi.org/10.1016/J.TCS.2020.11.006}
  {\path{doi:10.1016/J.TCS.2020.11.006}}.

\bibitem{DBLP:conf/spire/KutsukakeMNIBT20}
K.~Kutsukake, T.~Matsumoto, Y.~Nakashima, S.~Inenaga, H.~Bannai, M.~Takeda, On
  repetitiveness measures of {T}hue--{M}orse words, in: C.~Boucher, S.~V.
  Thankachan (Eds.), String Processing and Information Retrieval - 27th
  International Symposium, {SPIRE} 2020, Orlando, FL, USA, October 13-15, 2020,
  Proceedings, Vol. 12303 of Lecture Notes in Computer Science, Springer, 2020,
  pp. 213--220.
\newblock \href {https://doi.org/10.1007/978-3-030-59212-7_15}
  {\path{doi:10.1007/978-3-030-59212-7_15}}.

\bibitem{schaeffer2024stringattractorsautomaticsequences}
L.~Schaeffer, J.~Shallit, String attractors for automatic sequences (2024).
\newblock \href {http://arxiv.org/abs/2012.06840} {\path{arXiv:2012.06840}}.

\bibitem{mousavi2021automatictheoremprovingwalnut}
H.~Mousavi, Automatic theorem proving in {Walnut} (2021).
\newblock \href {http://arxiv.org/abs/1603.06017} {\path{arXiv:1603.06017}}.

\bibitem{DBLP:journals/ejc/WenW94}
Z.~Wen, Z.~Wen, Some properties of the singular words of the {Fibonacci} word,
  Eur. J. Comb. 15~(6) (1994) 587--598.
\newblock \href {https://doi.org/10.1006/EUJC.1994.1060}
  {\path{doi:10.1006/EUJC.1994.1060}}.

\bibitem{DBLP:conf/cpm/BanbaraBGKMO26}
M.~Banbara, H.~Bannai, P.~Guo, D.~K{\"{o}}ppl, T.~Mieno, Y.~Okamoto,
  \href{https://doi.org/10.4230/LIPIcs.CPM.2026.33}{The smallest string
  attractors of fibonacci and period-doubling words}, in: P.~Bille, N.~Prezza
  (Eds.), 37th Annual Symposium on Combinatorial Pattern Matching, {CPM} 2026,
  Copenhagen, Denmark, June 15-17, 2026, Vol. 369 of LIPIcs, Schloss Dagstuhl -
  Leibniz-Zentrum f{\"{u}}r Informatik, 2026, pp. 33:1--33:21.
\newblock \href {https://doi.org/10.4230/LIPICS.CPM.2026.33}
  {\path{doi:10.4230/LIPICS.CPM.2026.33}}.
\newline\urlprefix\url{https://doi.org/10.4230/LIPIcs.CPM.2026.33}

\bibitem{DBLP:conf/cpm/MienoIH22}
T.~Mieno, S.~Inenaga, T.~Horiyama, {RePair} grammars are the smallest grammars
  for {Fibonacci} words, in: H.~Bannai, J.~Holub (Eds.), 33rd Annual Symposium
  on Combinatorial Pattern Matching, {CPM} 2022, Prague, Czech Republic, June
  27-29, 2022, Vol. 223 of LIPIcs, Schloss Dagstuhl - Leibniz-Zentrum f{\"{u}}r
  Informatik, 2022, pp. 26:1--26:17.
\newblock \href {https://doi.org/10.4230/LIPICS.CPM.2022.26}
  {\path{doi:10.4230/LIPICS.CPM.2022.26}}.

\bibitem{DBLP:conf/dcc/LarssonM99}
N.~J. Larsson, A.~Moffat, Offline dictionary-based compression, in: Data
  Compression Conference, {DCC} 1999, Snowbird, Utah, USA, March 29-31, 1999,
  {IEEE} Computer Society, 1999, pp. 296--305.
\newblock \href {https://doi.org/10.1109/DCC.1999.755679}
  {\path{doi:10.1109/DCC.1999.755679}}.

\bibitem{DBLP:journals/jcta/CassaigneGRRSS24}
J.~Cassaigne, F.~Gheeraert, A.~Restivo, G.~Romana, M.~Sciortino, M.~Stipulanti,
  \href{https://doi.org/10.1016/j.jcta.2024.105936}{New string attractor-based
  complexities for infinite words}, J. Comb. Theory, Ser. {A} 208 (2024)
  105936.
\newblock \href {https://doi.org/10.1016/J.JCTA.2024.105936}
  {\path{doi:10.1016/J.JCTA.2024.105936}}.
\newline\urlprefix\url{https://doi.org/10.1016/j.jcta.2024.105936}

\bibitem{FaroMarinoPavone2026}
S.~Faro, F.~P. Marino, A.~Pavone, Smallest string attractors and minimal
  coverage certificates of {Thue--Morse} words (2026).
\newblock \href {http://arxiv.org/abs/2610.10950} {\path{arXiv:2610.10950}}.

\bibitem{KishiGUB26}
K.~Kishi, P.~Guo, C.~Urbina, H.~Bannai, On occurrence-preserving morphisms, in:
  P.~Bille, N.~Prezza (Eds.), 37th Annual Symposium on Combinatorial Pattern
  Matching, {CPM} 2026, Copenhagen, Denmark, June 15-17, 2026, Vol. 369 of
  LIPIcs, Schloss Dagstuhl - Leibniz-Zentrum f{\"{u}}r Informatik, 2026, pp.
  24:1--24:16.

\bibitem{IlieS11}
L.~Ilie, W.~F. Smyth, Minimum unique substrings and maximum repeats, Fundam.
  Informaticae 110~(1-4) (2011) 183--195.
\newblock \href {https://doi.org/10.3233/FI-2011-536}
  {\path{doi:10.3233/FI-2011-536}}.

\bibitem{DBLP:journals/ipl/Luca81}
A.~de~Luca, A combinatorial property of the {F}ibonacci words, Inf. Process.
  Lett. 12~(4) (1981) 193--195.
\newblock \href {https://doi.org/10.1016/0020-0190(81)90099-5}
  {\path{doi:10.1016/0020-0190(81)90099-5}}.

\bibitem{DBLP:journals/siamcomp/KnuthMP77}
D.~E. Knuth, J.~H. Morris, Jr., V.~R. Pratt, Fast pattern matching in strings,
  {SIAM} J. Comput. 6~(2) (1977) 323--350.
\newblock \href {https://doi.org/10.1137/0206024} {\path{doi:10.1137/0206024}}.

\bibitem{Lothaire_2002}
M.~Lothaire, Sturmian Words, Vol.~90 of Encyclopedia of Mathematics and its
  Applications, Cambridge University Press, 2002, Ch.~2, pp. 45--110.

\bibitem{DBLP:journals/jintseq/Fici15}
G.~Fici,
  \href{https://cs.uwaterloo.ca/journals/JIS/VOL18/Fici/fici5.html}{Factorizations
  of the {Fibonacci} infinite word}, J. Integer Seq. 18~(9) (2015) 15.9.3.
\newline\urlprefix\url{https://cs.uwaterloo.ca/journals/JIS/VOL18/Fici/fici5.html}

\bibitem{DBLP:conf/spire/BannaiFIKMN21}
H.~Bannai, M.~Funakoshi, T.~I, D.~K{\"{o}}ppl, T.~Mieno, T.~Nishimoto, A
  separation of {\(\gamma\)} and b via {T}hue--{M}orse words, in: T.~Lecroq,
  H.~Touzet (Eds.), String Processing and Information Retrieval - 28th
  International Symposium, {SPIRE} 2021, Lille, France, October 4-6, 2021,
  Proceedings, Vol. 12944 of Lecture Notes in Computer Science, Springer, 2021,
  pp. 167--178.
\newblock \href {https://doi.org/10.1007/978-3-030-86692-1_14}
  {\path{doi:10.1007/978-3-030-86692-1_14}}.

\bibitem{DBLP:journals/dam/Brlek89}
S.~Brlek, \href{https://doi.org/10.1016/0166-218X(92)90274-E}{Enumeration of
  factors in the thue-morse word}, Discret. Appl. Math. 24~(1-3) (1989) 83--96.
\newblock \href {https://doi.org/10.1016/0166-218X(92)90274-E}
  {\path{doi:10.1016/0166-218X(92)90274-E}}.
\newline\urlprefix\url{https://doi.org/10.1016/0166-218X(92)90274-E}

\bibitem{DBLP:journals/jda/RadoszewskiR12}
J.~Radoszewski, W.~Rytter, On the structure of compacted subword graphs of
  {T}hue-{M}orse words and their applications, J. Discrete Algorithms 11 (2012)
  15--24.
\newblock \href {https://doi.org/10.1016/J.JDA.2011.01.001}
  {\path{doi:10.1016/J.JDA.2011.01.001}}.

\bibitem{DBLP:conf/cpm/GuoK25}
P.~Guo, K.~Kishi, \href{https://doi.org/10.4230/LIPIcs.CPM.2025.16}{Net
  occurrences in {F}ibonacci and {T}hue-{M}orse words}, in: P.~Bonizzoni,
  V.~M{\"{a}}kinen (Eds.), 36th Annual Symposium on Combinatorial Pattern
  Matching, {CPM} 2025, Milan, Italy, June 17-19, 2025, Vol. 331 of LIPIcs,
  Schloss Dagstuhl - Leibniz-Zentrum f{\"{u}}r Informatik, 2025, pp.
  16:1--16:22.
\newblock \href {https://doi.org/10.4230/LIPICS.CPM.2025.16}
  {\path{doi:10.4230/LIPICS.CPM.2025.16}}.
\newline\urlprefix\url{https://doi.org/10.4230/LIPIcs.CPM.2025.16}

\bibitem{DBLP:conf/tacas/NicolF26}
J.~Nicol, M.~Frohme,
  \href{https://doi.org/10.1007/978-3-032-22749-2_20}{Deconstructing subset
  construction - reducing while determinizing}, in: S.~Junges, G.~Katz (Eds.),
  Tools and Algorithms for the Construction and Analysis of Systems - 32nd
  International Conference, {TACAS} 2026, Held as Part of the International
  Joint Conferences on Theory and Practice of Software, {ETAPS} 2026, Turin,
  Italy, April 11-16, 2026, Proceedings, Part {II}, Vol. 16506 of Lecture Notes
  in Computer Science, Springer, 2026, pp. 394--415.
\newblock \href {https://doi.org/10.1007/978-3-032-22749-2_20}
  {\path{doi:10.1007/978-3-032-22749-2_20}}.
\newline\urlprefix\url{https://doi.org/10.1007/978-3-032-22749-2_20}

\end{thebibliography}
\end{document}